\documentclass[11pt]{article}

\usepackage{amssymb}
\usepackage{amsmath}
\usepackage{amsthm}
\usepackage{pgf,acadpgf}

\long\def\commentout#1{}

\numberwithin{equation}{section}

\newtheoremstyle{slplain}
  {\topsep}
  {\topsep}
  {\slshape}
  {0pt}
  {\bfseries}
  {.}
  {0.5em}
  {}

\theoremstyle{slplain}
  \newtheorem{THM}{Theorem}[section]
  \newtheorem{LEM}[THM]{Lemma}
  
  \newtheorem{COR}[THM]{Corollary}

\theoremstyle{definition}

\newcommand{\Fraisse}{Fra\"\i ss\'e}

\newcommand\nlongrightarrow{\longrightarrow\kern -1.45em/\kern 0.9em}

\newcommand{\UNION}{\bigcup}

\renewcommand{\le}{\leqslant}
\renewcommand{\ge}{\geqslant}

\newcommand{\0}{\varnothing}

\renewcommand{\sec}{\cap}
\renewcommand{\phi}{\varphi}
\renewcommand{\epsilon}{\varepsilon}

\newcommand{\NN}{\mathbb{N}}

\newcommand{\QQ}{\mathbb{Q}}

\newcommand{\ZZ}{\mathbb{Z}}
\newcommand{\union}{\cup}
\newcommand{\restr}[2]{\hbox{$#1$}\hbox{$\upharpoonright$}_{#2}}

\newcommand{\Boxed}[1]{\mbox{$#1$}}
\newcommand{\id}{\mathrm{id}}

\newcommand{\out}{\mathrm{out}}
\newcommand{\high}{\mathrm{ht}}
\newcommand{\succs}{\mathrm{succ}}
\newcommand{\fin}{\mathrm{fin}}
\newcommand{\Fin}{\mathrm{Fin}}
\newcommand{\nat}{\mathrm{nat}}

\newcommand{\calA}{\mathcal{A}}

\newcommand{\calH}{\mathcal{H}}

\newcommand{\calP}{\mathcal{P}}

\newcommand{\calS}{\mathcal{S}}
\newcommand{\calT}{\mathcal{T}}

\newcommand{\tp}{\mathrm{tp}}

\newcommand{\Emb}{\mathrm{Emb}}
\newcommand{\im}{\mathrm{im}}

\newcommand{\spec}{\mathrm{spec}}
\newcommand{\tree}[1]{\langle#1\rangle}

\title{Big Ramsey spectra of countable chains}
\author{%
  Dragan Ma\v sulovi\'c\\
  University of Novi Sad, Faculty of Sciences\\
  Department of Mathematics and Informatics\\
  Trg Dositeja Obradovi\'ca 3, 21000 Novi Sad, Serbia\\
  e-mail: dragan.masulovic@dmi.uns.ac.rs}

\begin{document}
\maketitle

\begin{abstract}
    A big Ramsey spectrum of a countable chain (i.e.\ strict linear order) $C$ is a sequence
    of big Ramsey degrees of finite chains computed in $C$. In this paper we consider big Ramsey
    spectra of countable scattered chains. We prove that
    countable scattered chains of infinite Hausdorff rank do not have finite big Ramsey spectra,
    and that countable scattered chains of finite Hausdorff rank with bounded finite sums
    have finite big Ramsey spectra.
    Since big Ramsey spectra of all non-scattered countable chains are finite by results of Galvin, Laver and Devlin,
    in order to complete the characterization of countable chains with finite big Ramsey spectra (or degrees)
    one still has to resolve the remaining case of countable scattered chains of finite Hausdorff rank
    whose finite sums are not bounded.

  \bigskip


  \noindent \textbf{AMS Subj.\ Classification (2010): 05D10, 06A05}
\end{abstract}

\section{Introduction}

Ramsey's famous theorem:

\begin{THM}[Ramsey's Theorem \cite{Ramsey}]\label{fbrd-scat-all.thm.RAMSEY}
  For any $n \ge 1$ and an arbitrary coloring $\chi : \binom \omega n \to k$
  of $n$-element subsets of $\omega$ with $k \ge 2$ colors there exists a copy
  $M \subseteq \omega$ of $\omega$ which is \emph{monochromatic} in the following sense: $\chi(X) = \chi(Y)$
  for all $X, Y \in \binom M n$.
\end{THM}

\noindent
was published in 1930, but already in 1933 it was generalized to cardinals by Sierpi\'nski.
This marked the beginning of combinatorial set theory which is nowadays a deep and influential part of set
theory (see~\cite{williams}). In contrast to Ramsey theory which abounds with positive results of the form
``for any coloring there is a monochromatic copy'', the generalization of Ramsey's Theorem to cardinals
brought a plethora of negative or conditionally positive results of the form
``there is a complicated coloring such that no monochromatic copy exists''
or ``for any coloring there is a monochromatic copy provided we exclude a certain type of behavior''.

Scaling down to countable chains does not take us back to the realm where monochromatic copies dwell.
It is easy to construct a Sierpi\'nski-style coloring of $\binom \QQ 2$ with two colors
and with no monochromatic subchain isomorphic to~$\QQ$. However, Galvin showed in \cite{galvin1,galvin2} that
for every coloring $\chi : \binom \QQ 2 \to k$, $k \ge 2$, there is an \emph{oligochromatic} copy of $\QQ$
in the following sense: there is
a $U \subseteq \QQ$ order-isomorphic to $\QQ$ such that $\chi$ takes at most two colors on $\binom U 2$.
This observation was later generalized by Laver (see \cite[Theorem on p.~275]{erdos-hajnal}) who proved that
for every finite $n$ and every coloring $\chi : \binom \QQ n \to k$, $k \ge 2$, there is an
oligochromatic copy of $\QQ$ with at most $\mathbb{T}_n$ colors, where $\mathbb{T}_n$ depends only on~$n$.
In his thesis~\cite{devlin} Devlin actually managed to compute the numbers $\mathbb{T}_n$ and it turns out that
$\mathbb{T}_n = \tan^{(2n-1)}(0)$.

The integer $\mathbb{T}_n$ is referred to as the \emph{big Ramsey degree of~$n$ in~$\QQ$}.
This term was coined by Kechris, Pestov and Todor\v cevi\'c \cite{KPT} where big Ramsey degrees were
considered under this name in the context of structural Ramsey theory of \Fraisse\ limits.
In particular, an integer $T \ge 1$ is a \emph{big Ramsey degree of a finite chain $n$ in a chain $A$}
if it is the smallest positive integer such that
for every coloring $\chi : \binom A n \to k$ where $k \ge 2$
there is a $U \subseteq A$ order-isomorphic to $A$ such that $\chi$ takes at most $T$
colors on $\binom U n$. If no such $T$ exists we say that \emph{$n$ does not have big Ramsey degree in $A$}.
We denote the big Ramsey degree of $n$ in $A$ by $T(n, A)$, and write
$T(n, A) = \infty$ if $n$ does not have the big Ramsey degree in~$A$.

Clearly, for every $n \in \NN$ there is, up to isomorphism, only one chain of length~$n$.
Hence, for any chain $A$ it is convenient to consider the \emph{big Ramsey spectrum of $A$}:
$$
  \spec(A) = (T(1, A), T(2, A), T(3, A), \ldots) \in (\NN \union \{\infty\})^\NN.
$$
A chain $A$ \emph{has finite big Ramsey spectrum} if $T(n, A) < \infty$ for all $n \ge 1$,
that is, if $\spec(A) \in \NN^\NN$.
In this parlance the Ramsey's theorem and the results of Galvin and Devlin take the following form:

\begin{THM}
  $(a)$ (Ramsey \cite{Ramsey}) $\spec(\omega) = (1, 1, 1, \ldots)$.
  
  $(b)$ (Galvin \cite{galvin1,galvin2}) $T(2, \QQ) = 2$.

  $(c)$ (Devlin \cite{devlin}) $\spec(\QQ) = (\mathbb{T}_1, \mathbb{T}_2, \mathbb{T}_3, \ldots)$, which coincides with
  the OEIS sequence A000182.
\end{THM}

For a chain $A$ let $A^*$ denote $A$ with the order reversed. It is obvious that
$\spec(A) = \spec(A^*)$ for all chains $A$. In particular, $\spec(\omega) = \spec(\omega^*)$.
Interestingly, $\omega$ and $\omega^*$ are the only countable chains whose spectrum is $(1, 1, 1, \ldots)$~\cite{rosenstein}.
We thus get the following strengthening of Ramsey's Theorem:

\begin{THM}[{\cite[Corollary 11.4]{rosenstein}}] Let $A$ be a countable chain.

  $(a)$ $T(2, A) = 1$ if and only if $A \cong \omega$ or $A \cong \omega^*$.
  
  $(b)$ Consequently, $\spec(A) = (1, 1, 1, \ldots)$ if and only if $A \cong \omega$ or $A \cong \omega^*$.
\end{THM}

It is very easy to show that if $A$ and $B$ are chains such that
$A$ embeds into $B$ and $B$ embeds into $A$ then $\spec(A) = \spec(B)$ for all $n \ge 1$.
Devlin's result, therefore, immediately applies to any non-scattered countable chain
(recall that a countable chain is \emph{scattered} if it does \emph{not} embed~$\QQ$,
otherwise it is \emph{non-scattered}):

\begin{THM}[Devlin \cite{devlin}]\label{fbrd-scat-all.thm.scat-devlin}
  If $A$ is a non-scattered countable chain then $\spec(A) = \spec(\QQ) = (\mathbb{T}_1, \mathbb{T}_2, \mathbb{T}_3, \ldots)$.
\end{THM}

Not much is known about big Ramsey spectra of scattered chains.
One of the most notable results in this direction was proved by R.~Laver:

\begin{THM}[Laver \cite{laver-decomposition}]\label{fbrd-scat-all.thm.LAVER}
  $T(1, S) < \infty$ for every scattered chain~$S$.
\end{THM}

In this paper we build on Laver's result by proving that countable scattered chains of infinite Hausdorff rank
do not have finite big Ramsey spectra, and prove that countable scattered chains of finite Hausdorff rank
\emph{with bounded finite sums} (see below for the definition) have finite big Ramsey spectra.
At the moment we are unable to resolve the remaining case of countable scattered chains of finite Hausdorff rank
whose finite sums are not bounded.

The paper is organized as follows.
In Section~\ref{fbrd-scat-all.sec.prelim} we recall some standard notions and notation.
In Section~\ref{fbrd-scat-all.sec.monotonicity} we prove that big Ramsey spectra of countable chains
are non-decreasing (where, as usual, we take $n < \infty$ for all $n \in \NN$).
In Section~\ref{fbrd-scat-all.sec.inf-rank} we prove
that big Ramsey spectra of countable scattered chains of infinite Hausdorff rank 
take the form $(n, \infty, \infty, \infty, \ldots)$ for some $n \in \NN$ (Theorem~\ref{fbrd-scat-all.thm.infty}).
In Section~\ref{fbrd-scat-all.sec.fin-rank} we focus on countable scattered chains of finite Hausdorff rank.
We identify a class of chains with \emph{bounded finite sums} (to be defined below) and prove
such scattered chains have finite big Ramsey spectra
(Theorem~\ref{fbrd-scat-all.thm.finite-rank}). As a corollary we prove that if $S$ is a chain of finite
Hausdorff rank satisfying $T(1, S) = 1$ then $\spec(S)$ is finite.

In both cases we rely on a Ramsey-type result and an appropriate representation of countable scattered chains.
Whereas in Section~\ref{fbrd-scat-all.sec.inf-rank} we use Galvin's result about square bracket partition relation
and work bottom-up using a ``semantic representation'' of scattered chains based on condensations,
in Section~\ref{fbrd-scat-all.sec.fin-rank} we use Laver's analysis of scattered chains from~\cite{laver-decomposition}
and an infinite version of the Product Ramsey Theorem,
and work top-down using a syntactic representation of scattered chains based on trees.
We close the paper with an open problem.

\section{Preliminaries}
\label{fbrd-scat-all.sec.prelim}

In a partially ordered set $(A, \Boxed\le)$ let $[a, b]_A = \{ x \in A : a \le x \le b \}$.
An \emph{interval of $A$} is a subset $I \subseteq A$ such that $[x, y]_A \subseteq I$ for all $x, y \in I$.
If we wish to stress that $A$ and $B$ are isomorphic as ordered sets we shall say that they are
\emph{order-isomorphic} and write $A \cong B$.

A \emph{chain (or a strict linear order)} is a pair $(A, \Boxed<)$ where $\le$ is a partial order on~$A$ with
no incomparable elements.
For a well-ordered set $(A, \Boxed<)$ let $\tp(A, \Boxed<)$ denote the \emph{order type} of $(A, \Boxed<)$,
that is, the unique ordinal $\alpha$ which is order-isomorphic to $(A, \Boxed<)$.
As usual, $\NN = \{1, 2, 3, \ldots \}$ is the chain of all the positive integers with the usual ordering,
$\ZZ = \{\ldots, -2, -1, 0, 1, 2, \ldots\}$ is the chain of all the integers with the usual ordering,
and $\QQ$ is the chain of all the rationals with the usual ordering. The order type of $\ZZ$ will be denoted by~$\zeta$.
Every integer $n \in \NN$ can be thought as a finite chain $0 < 1 < \ldots < n-1$.

Let $(A, \Boxed<)$ be a chain and assume that for each $a \in A$ we have a chain $(B_a, \Boxed{<_a})$.
Then the \emph{(indexed) sum of chains} $\sum_{a \in A} B_a$ is the chain on
$\bigcup_{a \in A} (\{a\} \times B_a)$ where the linear order $\prec$ is defined \emph{lexicographically}:
$(a, b) \prec (a', b')$ iff $a < a'$, or $a = a'$ and $b \mathrel{<_a} b'$. The \emph{product} of chains
$A$ and $B$ is the chain $\sum_{a \in A} B$. Instead of $\sum_{i \in n} B_i$ we shall write $B_0 + B_1 + \ldots + B_{n-1}$.

For the sake of simplicity we use the same notation for the operations on chains and for the
corresponding operations on ordinals. Moreover, in some
proofs we shall move freely between $\sum_{\xi \in \alpha} \beta_\xi$ and $\left(\UNION_{\xi \in \alpha} \{\xi\} \times \beta_\xi , \Boxed{\prec}\right)$,
and analogously for products. We believe that the context will always be sufficient to enable the correct parsing of the symbols.

The class of chains can be preordered by the embeddability relation in a usual way: for chains $A$ and $B$ we write
$A \hookrightarrow B$ to denote that there is an embedding from $A$ to $B$, and we write $A \equiv B$ if
$A \hookrightarrow B$ and $B \hookrightarrow A$.

A chain $A$ is \emph{scattered} if $\QQ \not\hookrightarrow A$; otherwise it is \emph{non-scattered}.
In 1908 Hausdorff published a structural characterization of scattered chains~\cite{hausdorff-scat}, which was
rediscovered by Erd\H os and Hajnal in their~1962 paper~\cite{erdos-hajnal-scat}.
We shall now present Hausdorff's characterization of countable scattered chains. Define a sequence $\calH_\alpha$ of chains
indexed by ordinals as follows:
\begin{itemize}
\item
  $\calH_0 = \{0, 1\}$;
\item
  for an ordinal $\alpha > 0$ let
  $
    \calH_\alpha = \{\sum_{i \in \ZZ} S_i : S_i \in \UNION_{\beta < \alpha} \calH_\beta \text{ for all } i \in \ZZ \}
  $.
\end{itemize}

\begin{THM}[Hausdorff~\cite{hausdorff-scat}]
  For each ordinal $\alpha$ the elements of $\calH_\alpha$ are countable scattered chains. Conversely,
  for every countable scattered chain $S$ there is an ordinal $\alpha$ such that $S \in \calH_\alpha$.
\end{THM}

The least ordinal $\alpha$ such that $\calH_\alpha$ contains a countable scattered chain $S$ is referred to as
the \emph{Hausdorff rank of $S$} and denoted by $r_H(S)$.
A countable scattered chain $S$ has \emph{finite Hausdorff rank} if $r_H(S) < \omega$;
otherwise it has \emph{infinite Hausdorff rank}.

Let $C$ be a chain and $n$ a finite chain. Then the set $\binom Cn$ of all the $n$-element subsets of $C$
clearly corresponds to the set $\Emb(n, C)$ of all the embeddings $n \hookrightarrow C$. We sometimes find it more convenient
to formally introduce big Ramsey degrees as follows. For chains $A$, $B$, $C$ and integers $k \ge 2$ and $t \ge 1$
we write $C \longrightarrow (B)^{A}_{k, t}$ to denote that for every $k$-coloring $\chi : \Emb(A, C) \to k$
there is an embedding $w \in \Emb(B, C)$ such that $|\chi(w \circ \Emb(A, B))| \le t$.
For a chain $C$ and a finite chain $n$ we say that $n$ has \emph{finite big Ramsey degree in $C$}
if there exists a positive integer $t$ such that for each $k \ge 2$ we have that
$C \longrightarrow (C)^{n}_{k, t}$.
Equivalently, a finite chain $n$ has finite big Ramsey degree in a chain $C$
if there exists a positive integer $t$ such that for every $k \ge 2$ and every
$k$-coloring $\chi : \Emb(n, C) \to k$ there is a $U \subseteq C$ order-isomorphic to $C$
such that $|\chi(\Emb(n, U))| \le t$.
The least such $t$ is then denoted by $T(n, C)$. If such a $t$ does not exist
we say that $A$ \emph{does not have finite big Ramsey degree in $C$} and write
$T(A, C) = \infty$. The sequence
$$
  \spec(C) = (T(1, C), T(2, C), T(3, C), \ldots) \in (\NN \union \{\infty\})^\NN
$$
is referred to as the \emph{big Ramsey spectrum of $C$}.
We say that $C$ \emph{has finite big Ramsey spectrum}
if $\spec(C) \in \NN^\NN$. For the sake of convenience, for any chain $C$ we let $T(0, C) = 1$ by definition.
It is easy to see that $C_1 \equiv C_2$ implies $\spec(C_1) = \spec(C_2)$.

\section{Monotonicity}
\label{fbrd-scat-all.sec.monotonicity}

Big Ramsey degrees in Fra\"\i ss\'e limits are monotonous in the following sense (see \cite{zucker-bigrd}).
Let $F$ be a countable \Fraisse\ limit in a relational language and let $A$ and $B$ be finite substructures of $F$.
Then $A \hookrightarrow B$ implies that $T(A, F) \le T(B, F)$. In this section we prove that the same holds for countable
chains. Consequently, the big Ramsey spectrum of any countable chain is a nondecreasing sequence of elements of
$\NN \union \{\infty\}$ where, of course, we take $\infty$ to be larger than any integer. Since in the context of arbitrary
chains we cannot rely on ultrahomogeneity, we shall proceed by discussing the structure of countable chains.

\begin{LEM}\label{fbrd-scat-all.lem.2-inf-n-inf}
  Let $A$ be an infinite chain and $m, n \in \NN$ such that $2 \le m \le n$. If $T(n, A) < \infty$ then $T(m, A) < \infty$. 
\end{LEM}
\begin{proof}
  Let $T(n, A) = t \in \NN$ and let us show that $T(m, A) \le t \cdot \binom nm$.
  Take any $k \ge 2$ and any $\chi : \Emb(m, A) \to k$. Define
  $$
    \chi' : \Emb(n, A) \to \calP(k)
  $$
  by
  $$
    \chi'(f) = \chi(f \circ \Emb(m, n)) \subseteq k.
  $$
  Since $T(n, A) = t$, there is an $A' \subseteq A$ order-isomorphic to $A$ such that
  $$
    |\chi'(\Emb(n, A'))| \le t.
  $$
  Therefore,
  \begin{align*}
    \chi'(\Emb(n, A')) &= \{ \chi'(f) : f \in \Emb(n, A') \}\\
                       &= \{ \chi(f \circ \Emb(m, n)) : f \in \Emb(n, A') \}
  \end{align*}
  has at most $t$ elements, so there exist not necessarily distinct $g_0, g_1, \ldots, g_{t-1} \in \Emb(n, A')$
  such that
  \begin{equation}\label{fbrd-scat-all.eq.2-n}
    \{ \chi(f \circ \Emb(m, n)) : f \in \Emb(n, A') \} = \{ \chi(g_i \circ \Emb(m, n)) : i < t \}.
  \end{equation}
  On the other hand, it is easy to see that $\Emb(m, A') = \Emb(n, A') \circ \Emb(m, n)$
  because $A'$, as a chain isomorphic to~$A$, is infinite. So,
  \begin{align*}
    \chi(\Emb(m, A'))
    &=\textstyle \chi(\Emb(n, A') \circ \Emb(m, n))\\
    &=\textstyle \chi(\UNION_{f \in \Emb(n, A')} f \circ \Emb(m, n))\\
    &=\textstyle \UNION_{f \in \Emb(n, A')} \chi(f \circ \Emb(m, n))\\
    &=\textstyle \UNION_{i < t} \chi(g_i \circ \Emb(m, n)) && \text{because of \eqref{fbrd-scat-all.eq.2-n}.}
  \end{align*}
  Therefore,
  $$
    |\chi(\Emb(m, A'))| \le \sum_{i < t} |\chi(g_i \circ \Emb(m, n))|
    \le t \cdot \binom nm,
  $$
  because $|\chi(g_i \circ \Emb(m, n))| \le \binom nm$.
\end{proof}

\begin{LEM}\label{fbrd-scat-all.lem.no-max}
  Let $A$ be a chain with no maximal element.
  Then $m \le n$ implies $T(m, A) \le T(n, A)$ for all $m, n \in \NN$.
\end{LEM}
\begin{proof}
  Let $T(n, A) = t \in \NN$.
  Take any $k \ge 2$ and let $\chi : \Emb(m, A) \to k$ be a coloring.
  Define $\chi' : \Emb(n, A) \to k$ by $\chi'(h) = \chi(\restr hm)$.
  Then there is an $A' \subseteq A$ order-isomorphic to $A$ such that
  $|\chi'(\Emb(n, A'))| \le t$. Since $A'$ is a chain with no maximal element, every
  $m$-element subchain of $A'$ can be extended to an $n$-element subchain of $A'$, whence
  $\chi(\Emb(m, A')) \subseteq \chi'(\Emb(n, A'))$. Therefore,
  $|\chi(\Emb(m, A'))| \le t$.
\end{proof}

\begin{LEM}\label{fbrd-scat-all.lem.B+omega*}
  Let $A$ be a chain such that $A = B + \omega^*$ for some chain $B$.
  Then $m \le n$ implies $T(m, A) \le T(n, A)$ for all $m, n \in \NN$.
\end{LEM}
\begin{proof}
  Without loss of generality we may assume that $B \sec \omega^* = \0$.
  Fix $m, n \in \NN$ such that $m \le n$. Let $f : m \hookrightarrow n$ be the inclusion
  $f(i) = i$, and let $g : A \hookrightarrow A$ be
  the self-embedding of $A$ where $g(b) = b$ for all $b \in B$ and $g(i) = i + n$
  for all $i \in \omega^*$. Because $g$ ``leaves enough room towards the end of the chain''
  it is easy to show that $g \circ \Emb(m, A) \subseteq \Emb(n, A) \circ f$.

  Let $T(n, A) = t \in \NN$.
  Take any $k \ge 2$ and let $\chi : \Emb(m, A) \to k$ be a coloring.
  Define $\chi' : \Emb(n, A) \to k$ by $\chi'(h) = \chi(h \circ f)$.
  Then there is a $w : A \hookrightarrow A$ such that $|\chi'(w \circ \Emb(n, A))| \le t$.
  The definition of $\chi'$ then yields that
  $|\chi(w \circ \Emb(n, A) \circ f)| \le t$. Therefore,
  $|\chi(w \circ g \circ \Emb(m, A))| \le t$ because $g \circ \Emb(m, A) \subseteq \Emb(n, A) \circ f$.
\end{proof}

\begin{LEM}\label{fbrd-scat-all.lem.B+r}
  Let $A$ be a chain such that $A = B + r$ for some $r \in \NN$ and some chain $B$ with no maximal element.
  Assume that $B \sec r = \0$ where $r = \{0, 1, \ldots, r-1\}$,
  and that there exists an embedding $\hat g : A \hookrightarrow A$ such that $\hat g(0) \notin r$.
  Then $m \le n$ implies $T(m, A) \le T(n, A)$ for all $m, n \in \NN$.
\end{LEM}
\begin{proof}
  Let $\hat g(0) = b_0 \in B$. Since $B$ does not
  have the maximal element there exist $b_1, \ldots, b_{r-1} \in B$ such that
  $b_0 < b_1 < \ldots < b_{r-1}$. Let $g : A \hookrightarrow A$ be
  the self-embedding of $A$ where $g(b) = \hat g(b)$ for all $b \in B$ and $g(i) = b_i$ for all $i < r$.
  Then $g$ is clearly a self embedding of $A$. Let $f : m \hookrightarrow n$ be the inclusion
  $f(i) = i$. As in the proof of Lemma~\ref{fbrd-scat-all.lem.B+omega*} it is easy to show that
  $g \circ \Emb(m, A) \subseteq \Emb(n, A) \circ f$ because $g$ ``leaves enough room towards the end of the chain''.
  We can now simply repeat the argument of Lemma~\ref{fbrd-scat-all.lem.B+omega*} to conclude the proof.
\end{proof}

Let $f : n \hookrightarrow B + r$ be an embedding, where $n, r \in \NN$ and $B$ is a chain such that
$B \sec r = \0$. Then $\tp(f) = \im(f) \sec r$ will be referred to as the \emph{type of $f$}. (For a set map $f : A \to B$
by $\im(f)$ we denote the image of $f$, that is, the set $\{f(a) : a \in A\}$.)
Given a type $\tau \subseteq r$, let
$$
  \Emb_\tau(n, B + r) = \{ f \in \Emb(n, B + r) : \tp(f) = \tau \}.
$$

\begin{LEM}\label{fbrd-scat-all.lem.1-additive}
  Let $n, r \in \NN$ and let $B$ be a chain such that $B \sec r = \0$.
  For every type $\tau \subseteq r$ with $|\tau| \le n$, every $k \ge 2$ and every coloring
  $
    \chi : \Emb_\tau(n, B + r) \to k
  $
  there is a $U \subseteq B$ order-isomorphic to $B$ such that
  $
    |\chi(\Emb_\tau(n, U + r))| \le T(n - |\tau|, B)
  $.
\end{LEM}
\begin{proof}
  Take an type $\tau \subseteq r$ such that $|\tau| \le n$ and assume that $T(n - |\tau|, B) < \infty$.
  If $|\tau| = n$ then for every $U \subseteq B$ we have that
  $|\Emb_\tau(n, U + r)| = 1$, whence $|\chi(\Emb_\tau(n, U + r))| = 1 = T(0, B)$.
  So, let $s = |\tau| < n$ and let
  $
    \Phi : \Emb_\tau(n, B + r) \to \Emb(n - s, B)
  $
  be the bijection that takes $f \in \Emb_\tau(n, B + r)$ to $\restr{f}{n-s} \in \Emb(n - s, B)$.
  Fix a $k \ge 2$ and a coloring
  $
    \chi : \Emb_\tau(n, B + r) \to k.
  $
  Let
  $
    \chi' : \Emb(n - s, B) \to k
  $
  be the coloring defined by $\chi'(f) = \chi(\Phi^{-1}(f))$. Then
  there is a $U \subseteq B$ order-isomorphic to $B$ such that
  $
    |\chi'(\Emb(n - s, U))| \le T(n - s, B)
  $.
  But then it easily follows that $|\chi(\Emb_\tau(n, U + r))| \le T(n - s, B)$.
\end{proof}

\begin{LEM}\label{fbrd-scat-all.lem.A+m}
  Let $B$ be a chain with no maximal element and let $r \in \NN$.
  Assume that $B \sec r = \0$ and that $\restr g r = \id_r$
  for every embedding $g : B + r \hookrightarrow B + r$.

  $(a)$ If $T(n, B + r) < \infty$ then $T(n - j, B) < \infty$ for all $n \in \NN$ and $0 \le j \le \min\{n,r\}$.

  $(b)$ If $T(n, B + r) < \infty$ then  $T(n, B + r) = \sum_{j=0}^{\min\{n,r\}} \binom rj \cdot T(n - j, B)$.
\end{LEM}
\begin{proof}
  $(a)$ Assume that $T(n - j, B) = \infty$ for some $0 \le j \le \min\{n,r\}$ and let us show that
  $T(n, B + r) = \infty$ by showing that $T(n, B + r) \ge t$ for every $t \in \NN$.
  Fix a $t \in \NN$. Because $T(n - j, B) = \infty$ there is a coloring $\chi : \Emb(n - j, B) \to k$ for some $k \ge t$
  such that $|\chi(w \circ \Emb(n - j, B))| \ge t$ for every $w : B \hookrightarrow B$.
  Define $\chi' : \Emb(n, B + r) \to k$ as follows:
  $$
    \chi'(f) = \begin{cases}
      \chi(\restr f {n - j}), & |\tp(f)| = j,\\
      0, & \text{otherwise}.
    \end{cases}
  $$
  Take any $g : B + r \hookrightarrow B + r$. Clearly, $\restr g B : B \hookrightarrow B$.
  Let us show that
  $$
    \chi'(g \circ \Emb(n, B + r)) \supseteq \chi(\restr gB \circ \Emb(n - j, B)).
  $$
  Take any $f \in \Emb(n - j, B)$ and let $h : j \to r$ be the inclusion $i \mapsto i$. Put
  $f' = f + h : n \hookrightarrow B + r$. Since $|\tp(f')| = j$ we have that
  $$
    \chi'(g \circ f') = \chi(\restr{(g \circ f')}{n - j}) = \chi(\restr gB \circ \restr{f'}{n - j}) = \chi(\restr gB \circ f).
  $$
  Therefore,
  $
    |\chi'(g \circ \Emb(n, B + r))| \ge |\chi(\restr gB \circ \Emb(n - j, B))| \ge t
  $.

  \medskip
  
  $(b)$ 
  Fix an $n \in \NN$ and assume that $T(n, B + r) < \infty$.
  Then $T(n - j, B) < \infty$ for all $0 \le j \le \min\{n,r\}$ (by $(a)$).
  Let $Q = \{ \tau \subseteq r : |\tau| \le n \}$ be the set of all the types realized by
  members of $\Emb(n, B + r)$. Let $Q = \{\tau_0, \tau_1, \ldots, \tau_{t-1}\}$ so that $|Q| = t$.
  Note that $t = \sum_{j=0}^{\min\{n,r\}} \binom rj$.

  Fix a $k \ge 2$ and a coloring
  $
    \chi : \Emb(n, B + r) \to k
  $.
  By Lemma~\ref{fbrd-scat-all.lem.1-additive} there is a $U_0 \subseteq B$ order-isomorphic to $B$ such that
  $$
    |\chi(\Emb_{\tau_0}(n, U_0 + r))| \le T(n - |\tau_0|, B).
  $$
  By the same lemma for each $j \in \{1, \ldots, t-1\}$ we then inductively obtain
  a $U_j \subseteq U_{j-1}$ order-isomorphic to $U_{j-1}$ (and hence to $B$) such that
  $$
    |\chi(\Emb_{\tau_j}(n, U_j + r))| \le T(n - |\tau_j|, B).
  $$
  Then, using the fact that $U_{t-1} \subseteq U_j$ we have that
  \begin{align*}
    |\chi(\Emb(n, U_{t-1} + r))|
    &= \sum_{j < t} |\chi(\Emb_{\tau_j}(n, U_{t-1} + r))|\\
    &\le \sum_{j < t} |\chi(\Emb_{\tau_j}(n, U_{j} + r))|\\
    &\le \sum_{j < t} T(n - |\tau_{j}|, B) \le \sum_{j=0}^{\min\{n,r\}} \binom rj \cdot T(n - j, B).
  \end{align*}

  In order to conclude the proof we have to show that there exists a coloring
  $\chi : \Emb(n, B + r) \to k$ where $k \ge \sum_{j=0}^{\min\{n,r\}} \binom rj \cdot T(n - j, B)$
  such that $|\chi(w \circ \Emb(n, B + r))| \ge \sum_{j=0}^{\min\{n,r\}} \binom rj \cdot T(n - j, B)$
  for every embedding $w : B + r \hookrightarrow B + r$.

  Since $T(n - j, B)$ is the big Ramsey degree of $n - j$ in $B$, for every $0 \le j \le \min\{n,r\}$ there is
  a coloring $\chi_j : \Emb(n-j, B) \to k_j$ where $k_j \ge T(n - j, B)$ such that
  $|\chi_j(v \circ \Emb(n - j, B))| \ge T(n - j, B)$ for every embedding $v : B \hookrightarrow B$.
  Define
  $$
    \chi : \Emb(n, B + r) \to \UNION_{\tau \in Q} \{\tau\} \times k_{|\tau|}
  $$
  as follows: for an $f \in \Emb(n, B + r)$ let $\tau = \tp(f)$ and $j = |\tau|$, and then put
  $$
    \chi(f) = (\tau, \chi_j(\restr{f}{n - j})).
  $$

  Take any embedding $w : B + r \hookrightarrow B + r$. By the assumption we know that $\restr wr = \id_r$.
  Clearly, $\chi(w \circ \Emb(n, B + r)) = \UNION_{\tau \in Q} \chi(w \circ \Emb_\tau(n, B + r))$. Let us show that
  this is a disjoint union.
  
  Take any $f \in \Emb(n, B + r)$. Note first that $\tp(w \circ f) = \tp(f)$
  because $\restr wr = \id_r$. For $j = |\tp(f)|$ we then have
  $$
    \chi(w \circ f) = (\tp(f), \chi_j(\restr{(w \circ f)}{n - j})) = (\tp(f), \chi_j(\restr{w}{B} \circ \restr{f}{n - j})).
  $$
  The claim now follows immediately, because the first component of $\chi(w \circ f)$ is $\tp(f)$.
  
  Consequently, $|\chi(w \circ \Emb(n, B + r))| = \sum_{\tau \in Q} |\chi(w \circ \Emb_\tau(n, B + r))|$.
  Now, take any $\tau \in Q$ and let $j = |\tau|$. Then
  \begin{align*}
    \chi(w \circ \Emb_\tau(n, B + r))
    &= \{ (\tau, \chi_j(\restr wB \circ \restr{f}{n-j})) : f \in \Emb_\tau(n, B + r)\}\\
    &= \{ (\tau, \chi_j(\restr wB \circ f')) : f' \in \Emb(n - j, B)\}.
  \end{align*}
  Therefore,
  \begin{align*}
    |\chi(w \circ \Emb(n, B + r))|
    & = \sum_{\tau \in Q} |\chi(w \circ \Emb_\tau(n, B + r))|\\
    & = \sum_{\tau \in Q} |\chi_{|\tau|}(\restr wB \circ \Emb(n - |\tau|, B)))|\\
    & \ge \sum_{\tau \in Q} T(n - |\tau|, B) \text{\quad [by the choice of $\chi_{|\tau|}$]}\\
    & = \sum_{j=0}^{\min\{n,r\}} \binom rj T(n - j, B). \qedhere
  \end{align*}
\end{proof}

\begin{THM}\label{fbrd-scat-all.thm.monotonicity}
  Let $A$ be a countable chain and $m, n \in \NN$. If $m \le n$ then $T(m, A) \le T(n, A)$.
\end{THM}
\begin{proof}
  Case 1. If $A$ is a non-scattered chain then $\spec(A) = (\mathbb{T}_1, \mathbb{T}_2, \mathbb{T}_3, \ldots)$
  by Theorem~\ref{fbrd-scat-all.thm.scat-devlin}, and it is a well known fact that $\mathbb{T}_1 < \mathbb{T}_2 < \mathbb{T}_3 < \ldots$.
  
  \medskip
  
  Case 2. If $A$ has no maximal element then Lemma~\ref{fbrd-scat-all.lem.no-max} applies.
  
  \medskip
  
  Case 3. If $A = B + \omega^*$ for some chain $B$ then Lemma~\ref{fbrd-scat-all.lem.B+omega*} applies.

  \medskip

  Case 4. Assume that $A$ is a scattered chain with a maximal element, but $A = B + \omega^*$ for \emph{no} chain $B$.

  Then $A = B + r$ for some $r \in \NN$ and some chain $B$ with no maximal element. Without loss of generality
  we can assume that $B \sec r = \0$. If there is an embedding $g : A \hookrightarrow A$ such that $g(0) \in B$
  Lemma~\ref{fbrd-scat-all.lem.B+r} applies. Therefore, for the rest of the proof assume that
  for every embedding $g : A \hookrightarrow A$ we have that $\restr g r = \id_r$.
  
  If $T(n, A) = \infty$ the statement is trivially true.
  
  If $T(n, A) < \infty$ then $T(m, A) < \infty$ either because $m = 1$ in which case Theorem~\ref{fbrd-scat-all.thm.LAVER}
  applies, or $m \ge 2$ in which case Lemma~\ref{fbrd-scat-all.lem.2-inf-n-inf} applies.
  Anyhow, both $T(m, A)$ and $T(n, A)$ are finite, so by Lemma~\ref{fbrd-scat-all.lem.A+m} (recall that $A = B + r$),
  \begin{align*}
    T(m, B + r) &= \sum_{j=0}^{\min\{m,r\}} \binom rj \cdot T(m - j, B) \text{ and}\\
    T(n, B + r) &= \sum_{j=0}^{\min\{n,r\}} \binom rj \cdot T(n - j, B).
  \end{align*}
  Since $m \le n$ and $B$ is a chain with no maximal element, Lemma~\ref{fbrd-scat-all.lem.no-max} ensures that
  $
    T(m - j, B) \le T(n - j, B)
  $ for all $0 \le j \le \min\{m,r\}$.
  Therefore, $T(m, B + r) \le T(n, B + r)$. This completes the proof.
\end{proof}

\section{Countable scattered chains of infinite rank}
\label{fbrd-scat-all.sec.inf-rank}

In this section we prove that countable scattered chains of infinite Hausdorff rank do not have finite big Ramsey
spectra. We prove the result using a ``semantic representation'' of scattered chains based on condensations.

A map $f : A \to B$ between two chains is a \emph{homomorphism} if $x \le y \Rightarrow f(x) \le (y)$
for all $x, y \in A$. A \emph{condensation of $A$}~\cite{rosenstein} is a surjective homomorphism $c : A \to B$.
Note that any condensation of a scattered chain is scattered.

If $c : A \to B$ is a condensation then $\rho = \ker c$ is an equivalence relation whose classes are intervals of~$A$.
Conversely, for every equivalence relation $\rho$ whose classes are intervals of~$A$
the linear order carries from $A$ to $A / \rho$ in the obvious way and the natural quotient map
$\nat_{A, \rho} : A \to A / \rho$ given by $\nat_{A, \rho}(a) = [a]_\rho$
is a condensation ($[a]_\rho$ is the equivalence class of $a$ with respect to $\rho$.)

A condensation $c : A \to B$ is referred to as \emph{finite} if the following holds: $c(x) = c(y)$ if and only if
$[x, y]_A \union [y, x]_A$ is finite~\cite{rosenstein}.
It is easy to see that if $c_1 : A \to B_1$ and $c_2 : A \to B_2$ are finite condensations of $A$
then $B_1 \cong B_2 \cong A/\theta_A$ where $\theta_A$ is defined by
$$
  (x, y) \in \theta_A \text{\quad if and only if\quad} [x, y]_A \union [y, x]_A \text{ is finite.}
$$
Hence, up to isomorphism of codomains, there is a unique finite condensation
of $A$ that we refer to as \emph{the finite condensation of~$A$} and denote by $\fin_A$.
Note that $\fin_A = \nat_{A, \theta_A}$.
For each finite condensation $\fin_A : A \to A/\theta_A$ and each $y \in A/\theta_A$ the order-type of
$\fin_A^{-1}(y)$ is either $n$ for some $n \in \NN$, or $\omega$, or $\omega^*$ or~$\zeta$.
We say that $y \in A/\theta_A$ is a \emph{finitary point} if $\fin_A^{-1}(y)$ is finite; otherwise we
say that $y$ is an \emph{infinitary point}. Clearly, there do not exist $x, y \in A/\theta_A$
such that $[x, y]_{A/\theta_A} = \{x, y\}$ and both $x$ and $y$ are finitary points.
The following is an immediate consequence:

\begin{LEM}\label{fbrd-scat-all.lem.infinitary-points}
  Let $S$ be a countable scattered chain, let $\fin_S : S \to S'$ be the finite condensation of~$S$ where
  $S' = S / \theta_S$, and let $I \subseteq S'$ be an infinite interval of~$S'$.
  Then there are infinitely many infinitary points in~$I$.
\end{LEM}
\begin{proof}
  Suppose that there are only finitely many infinitary points in~$I$.
  Then some infinite subinterval $J \subseteq I$ contains no infinitary points.
  Since $J$ is scattered there exist $a, b \in J$
  such that $a < b$ and $[a, b]_{S'} = \{a, b\}$ --- contradiction.
\end{proof}

Let $A$ be a chain. For each ordinal $\alpha$ let us define 
an equivalence relation $\theta^{[\alpha]}$ on $A$, a chain $A^{[\alpha]} = A / \theta^{[\alpha]}$
and a condensation $\Fin_A^\alpha : A \to A^{[\alpha]}$ inductively as follows.
\begin{itemize}
\item
  Let $\theta^{[0]} = \{ (a,a) : a \in A \}$ and define $\Fin_A^0 : A \to A^{[0]}$ by $\Fin_A^0(a) = \{a\}$.
\item
  For a successor ordinal $\alpha = \beta + 1$ let $\Fin_A^\alpha = \fin_{A^{[\beta]}} \circ \Fin_A^\beta$ and
  $\theta^{[\alpha]} = \ker\,\Fin_A^\alpha$.
\item
  For a limit ordinal $\lambda$ let $\theta^{[\lambda]} = \UNION_{\alpha < \lambda} \theta^{[\alpha]}$
  and define $\Fin_A^\lambda : A \to A^{[\lambda]}$ by $\Fin_A^\lambda(a) = [a]_{\theta^{[\lambda]}}$.
\end{itemize}

We say that an ordinal $\alpha$ is the \emph{finite condensation rank} of a chain~$A$ and write $r_F(A) = \alpha$ if $\alpha$ is
the least ordinal such that $A^{[\alpha]} \cong 1$.
For every countable scattered chain $S$ the finite condensation rank $r_F(S)$ exists and
$r_F(S) = r_H(S)$~\cite{rosenstein}.

The proof that we present in this section heavily relies on a powerful result of Galvin about square bracket partition relations
which express strong counterexamples to ordinary partition relations.
For chains $C$, $B_0$, $B_1$, $B_2$, \ldots, and $n < \omega$ write
$$
  C \longrightarrow [B_0, B_1, B_2, \ldots]^n
$$
to denote that for every coloring $\chi : \Emb(n, C) \to \omega$ there is an $i < \omega$ and
a subchain $U \subseteq C$ such that $U \cong B_i$ and $i \notin \chi(\Emb(n, U))$.
Erd\H os and Hajnal note in~\cite[p.~275]{erdos-hajnal} that in 1971
Galvin proved the following:

\begin{THM}[Galvin 1971]\label{fbrd-scat.thm.bracket}
  If $S$ is a scattered chain that contains no uncountable well-ordered subsets then
  $$
    S \nlongrightarrow [\omega, \omega^2, \omega^2, \omega^3, \omega^3, \ldots, \omega^{\lfloor \frac{i+1}{2}\rfloor + 1},\ldots]^2 \qquad(i \in \omega).
  $$
\end{THM}

\noindent
A recent proof of Galvin's result can be found in~\cite{Todorcevic-OscInts}.

For notational convenience let $\omega^{(+)} = \omega$ and $\omega^{(-)} = \omega^*$.
For a finite sequence $\delta = (\delta_0, \delta_1, \ldots, \delta_{n-1}) \in \{\Boxed+, \Boxed-\}^n$
let
$$
  \omega^{(\delta)} = \omega^{(\delta_0)} \cdot \omega^{(\delta_1)} \cdot \ldots \cdot \omega^{(\delta_{n-1})}.
$$
Let $\alpha$ be an ordinal. For an $\alpha$-sequence $\delta = (\delta_i)_{i < \alpha} \in \{\Boxed+, \Boxed-\}^\alpha$
and $n < \alpha$ we let $\restr \delta n = (\delta_0, \delta_1, \ldots, \delta_{n-1})$.

\begin{LEM}\label{fbrd-scat-all.lem.2}
  Let $S$ be a countable scattered chain such that $r_H(S) \ge \omega$. There exists
  an $\omega$-sequence $\delta \in \{\Boxed+, \Boxed-\}^\omega = (\delta_0, \delta_1, \ldots)$
  such that $\omega^{(\restr \delta k)} \hookrightarrow S$ for all $k \ge 1$.
\end{LEM}
\begin{proof}
  Each of the chains $S \cong S^{[0]}$, $S^{[1]}$, $S^{[2]}$, \ldots, $S^{[n]}$, \ldots\ is
  a countably infinite scattered chain. Recall that $\fin_{S^{[j]}} : S^{[j]} \to S^{[j+1]}$, $j \ge 0$.
  For each $j \in \NN$ we shall label
  infinitary points of $S^{[j]}$ by elements of $\{\Boxed+, \Boxed-\}^j$ and along the way
  build a set $\calT$ of finite words over $\{\Boxed+, \Boxed-\}$ as follows. 
  To start the induction put the empty word $\epsilon$ in $\calT$ and
  label each infinitary point $y \in S^{[1]}$ by $+$ if the order type of $\fin_{S^{[0]}}^{-1}(y)$ is $\omega$ or $\zeta$;
  otherwise label the point by~$-$. Add all the labels assigned to infinitary points of $S^{[1]}$ to~$\calT$.
  Note that $\ell \in \calT$ means that $\omega^{(\ell)} \hookrightarrow S$.
  
  Assume that all the infinitary points of $S^{[j]}$ have been labelled by elements of $\{\Boxed+, \Boxed-\}^j$.
  Take any infinitary point $y \in S^{[j+1]}$. Since $\fin^{-1}_{S^{[j]}}(y)$ is an infinite interval of $S^{[j]}$,
  it contains infinitely many infinitary points (Lemma~\ref{fbrd-scat-all.lem.infinitary-points}).
  Each of the infinitary points in $\fin^{-1}_{S^{[j]}}(y)$ is labelled by one of the $2^j$ labels,
  so there is a label $\ell \in \{\Boxed+, \Boxed-\}^j$ which occurs infinitely many times in~$\fin^{-1}_{S^{[j]}}(y)$.
  If the order type of $\fin^{-1}_{S^{[j]}}(y)$ is $\omega$ or $\zeta$ label $y$ by $\ell\Boxed+$; otherwise label $y$ by $\ell\Boxed-$.
  Add all the labels assigned to infinitary points of $S^{[j+1]}$ to~$\calT$.
  Note again that $\ell \in \calT$ means that $\omega^{(\ell)} \hookrightarrow S$.
  
  The prefix ordering turns $\calT$ into an infinite (not necessarily full) binary tree, so by K\H onig's Lemma
  there is an infinite branch $\delta \in \{\Boxed+, \Boxed-\}^\omega = (\delta_0, \delta_1, \ldots)$.  Clearly, the construction
  ensures that $\omega^{(\restr \delta k)} \hookrightarrow S$ for all $k \ge 1$.
\end{proof}

\begin{THM}\label{fbrd-scat-all.thm.infty}
  Let $S$ be a countable scattered chain such that $r_H(S) \ge \omega$. Then
  $T(n, S) = \infty$ for every $2 \le n < \omega$.
\end{THM}
\begin{proof}
  Due to Lemma~\ref{fbrd-scat-all.lem.2-inf-n-inf} it suffices to show that $T(2, S) = \infty$.

  According to Lemma~\ref{fbrd-scat-all.lem.2} there exists
  an $\omega$-sequence $\delta \in \{\Boxed+, \Boxed-\}^\omega = (\delta_0, \delta_1, \ldots)$
  such that $\omega^{(\restr \delta k)} \hookrightarrow S$ for all $k \ge 1$.

  \medskip
  
  Case 1: The symbol $+$ occurs infinitely many times in $\delta$.
  
  Since $S$ is a countable scattered chain Theorem~\ref{fbrd-scat.thm.bracket} applies, so
  $$
    S \nlongrightarrow [\omega^{n_0}, \omega^{n_1}, \omega^{n_2}, \omega^{n_3}, \omega^{n_4}, \ldots]^2,
  $$
  where $n_i = \lfloor \frac{i+1}{2}\rfloor + 1$, $i \in \omega$. Therefore, there exists a coloring
  $\gamma : \Emb(2, S) \to \omega$ with the following property: for every $i < \omega$ and every
  subchain $H \subseteq S$ such that $H \cong \omega^{n_i}$ we have that $i \in \gamma(\Emb(2, H))$.

  Note that if $+$ appears $m$ times in $\restr \delta k$ then, clearly, $\omega^m \hookrightarrow \omega^{(\restr \delta k)}$. Analogously,
  if $-$ appears $m$ times in $\restr \delta k$ then $(\omega^*)^m \hookrightarrow \omega^{(\restr \delta k)}$.
  Since $+$ occurs infinitely many times in $\delta$ it follows that $\omega^i \hookrightarrow S$ for all $i \in \NN$.
  
  Now, take any $t \ge 2$ and consider the coloring $\chi_t : \Emb(2, S) \to t$ given by
  $\chi_t(f) = \min\{t-1, \gamma(f)\}$. Let $S'$ be an arbitrary subchain of $S$ order-isomorphic to $S$.
  Since $\omega^i \hookrightarrow S \cong S'$ for all $i \in \NN$,
  for every $i < t$ there is a subchain $H_i \subseteq S'$ order-isomorphic to $\omega^{n_i}$.
  By the construction of $\chi_t$ it then follows that $i \in \chi_t(\Emb(2, H_i)) \subseteq \chi_t(\Emb(2, S'))$.
  Therefore, $|\chi_t(\Emb(2, S'))| \ge t$.
  This concludes the proof that $T(2, S) = \infty$ in Case~1.

  \medskip
  
  Case 2: The symbol $+$ occurs only finitely many times in $\delta$.
  
  Then the symbol $-$ occurs infinitely many times in $\delta$, so
  $(\omega^*)^i \hookrightarrow S$ for all $i \in \NN$.
  Therefore, $\omega^i \hookrightarrow S^*$ for all $i \in \NN$.
  This time we apply Theorem~\ref{fbrd-scat.thm.bracket} to $S^*$ to conclude that
  $S^* \nlongrightarrow [\omega^{n_0}, \omega^{n_1}, \omega^{n_2}, \omega^{n_3}, \omega^{n_4}, \ldots]^2$,
  and as in Case~1 we conclude that $T(2, S^*) = \infty$. But it is easy to see that $T(2, S) = T(2, S^*)$
  for every chain~$S$. Therefore, $T(2, S) = \infty$.
\end{proof}

\section{Countable scattered chains of finite rank}
\label{fbrd-scat-all.sec.fin-rank}

In order to complete the characterization of countable chains with finite big Ramsey spectra
we still have to consider countable scattered chains of finite Hausdorff rank.
It is out conjecture that all such scattered chains have finite big Ramsey spectra.
To support the conjecture we will prove that the conjecture is true for a class of countable scattered
chains which includes indivisible countable scattered chains.
Our discussion heavily relies on Laver's deep analysis of the structure
of scattered chains presented in~\cite{laver-fraisse-conj} and~\cite{laver-decomposition}.
Note, however, that we use a different (and simpler) tree representation of scattered chains.

A \emph{rooted tree} is a pair $\tau = (T, v_0)$ where $T$ is a partially ordered set, $v_0 \in T$ is the \emph{root} of $T$
and $[v_0, x]_T$ is well-ordered for every $x \in T$. Maximal chains in~$T$ are called the \emph{branches} of $\tau$.
The \emph{height} of a rooted tree is the supremum of order-types of branches in $T$:
$$
  \high(\tau) = \sup\{\tp(b) : b \text{ is a branch in } \tau\}.
$$
For a vertex $x \in T$ let $\succs_\tau(x)$ be the set of all the immediate successors of $x$ and let
$\out_\tau(x) = \{(x, y) : y \in \succs_\tau(x)\}$ be the set of the \emph{outgoing edges}.
A vertex $x \in T$ is a \emph{leaf of $\tau$} if $\succs_\tau(x) = \0$.
Every finite branch starts at the root of the tree and ends in a leaf.

Let $\{b_\xi : \xi < \alpha\}$ be a set of branches of a rooted tree $\tau = (T, v_0)$. The \emph{subtree of $\tau$
induced by branches $b_\xi$, $\xi < \alpha$,} is the subtree of $\tau$ induced by the set of vertices $\UNION_{\xi < \alpha} b_\xi$.

A rooted tree $\tau = (T, v_0)$ is \emph{ordered} if $\out_\tau(x)$ is a chain for every $x \in T$.
If $\high(\tau) \le \omega$ then the linear orders on $\out_\tau(x)$, $x \in T$, uniquely determine a linear
ordering on the vertices of $T$: just traverse the tree using the breadth-first-search strategy. This means that we start
with the root $v_0$, then list the immediate successors of $v_0$ according to the ordering of $\out_\tau(v_0)$, and so on.
We refer to this ordering as the \emph{BFS-ordering of $\tau$}.

A \emph{labelled ordered rooted tree} is an ordered rooted
tree whose vertices are labelled by the elements of some set $L_v$,
and edges are labelled by the elements of some set $L_e$.
For a labelled ordered rooted tree $\tau$ by $L_v(\tau)$ we denote the set of vertex labels that appear in~$\tau$,
and by $L_e(\tau)$ we denote the set of edge labels that appear in~$\tau$.

Let $\tau$ be a labelled ordered rooted tree whose vertices are labelled by elements of $L_v$
and edges are labelled by elements of $L_e$, and let $U \subseteq L_e$.
By $\restr \tau U$ we denote the subtree of $\tau$ induced by all of its branches whose edge labels belong to~$U$.

Let us now recall Laver's analysis of the structure of scattered chains from~\cite{laver-fraisse-conj} and~\cite{laver-decomposition}.
A scattered chain $C$ is \emph{additively indecomposable} if $C \hookrightarrow A + B$ implies
$C \hookrightarrow A$ or $C \hookrightarrow B$ for all scattered chains $A$ and $B$.
Note that the empty chain $\0$ and the one-element chain 1 are additively indecomposable.

\begin{THM}[Laver~\cite{laver-fraisse-conj}]\label{fbrd-scat-spec.thm.laver-1}
  Every scattered chain is a finite sum of additively indecomposable scattered chains.
\end{THM}

The following result is implicit in the proof of~\cite[Theorem 2.11]{laver-decomposition}:

\begin{THM}[Laver~\cite{laver-decomposition}]\label{fbrd-scat-spec.thm.laver-2}
  Let $S$ be an additively indecomposable countable scattered chain. Then one of the following holds:
  \begin{itemize}
  \item
    there exist additively indecomposable countable scattered chains $R_i$, $i \in \omega$, such that
    $R_0 \hookrightarrow R_1 \hookrightarrow R_2 \hookrightarrow \ldots$ and either $S \equiv \sum_{i \in \omega} R_i$
    or $S \equiv \sum_{i \in \omega^*} R_i$; or
  \item
    there exists an $n \in \NN$ and additively indecomposable countable scattered chains $R_{ij}$, $i \in \omega$, $0 \le j \le n$,
    such that $R_{0j} \hookrightarrow R_{1j} \hookrightarrow R_{2j} \hookrightarrow \ldots$ for all $0 \le j \le n$
    and either $S \equiv \sum_{i \in \omega} (R_{i0} + R_{i1} + \ldots + R_{in})$
    or $S \equiv \sum_{i \in \omega^*} (R_{i0} + R_{i1} + \ldots + R_{in})$.
  \end{itemize}
\end{THM}

Let us now define a family of sets $\calA_n$, $n \in \omega$, of labelled ordered rooted trees and the scattered chains they encode.
Let $L_v = \{0, 1, \Boxed+, \omega, \omega^*\}$ be the set of vertex labels and let
$L_e = \omega \union \{\iota_n : n \in \omega\}$ be the set of edge labels.
Let $\calA_0 = \{ \bullet0, \bullet1 \}$ be the set whose elements are
single-vertex trees $\bullet0$ (a vertex labelled by $0$) and $\bullet1$ (a vertex labelled by $1$);
the chains these trees encode are $\|\bullet0\| = \0$ -- the empty chain, and $\|\bullet1\| = 1$ -- the trivial one-element chain;

Assume that $\calA_i$ have been defined for all $i < m$ and let us define three operations on trees as follows:
\begin{itemize}
\item
  for $n \in \NN$ and $\tau_0, \ldots, \tau_n \in \bigcup_{i < m} \calA_i$ let
  $\sigma$ be the tree whose root is labelled by $+$,
  edges going out of the root are labelled by $\iota_0, \ldots, \iota_{n}$ and are ordered that way, and each edge $\iota_k$
  leads to a subtree isomorphic to $\tau_k$, $0 \le k \le n$:
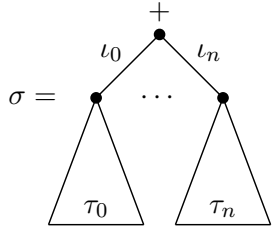
  \begin{center}
\begin{pgfpicture}
  \pgfsetxvec{\pgfpoint{\acadpgfunit}{0pt}}
  \pgfsetyvec{\pgfpoint{0pt}{\acadpgfunit}}
  \pgfsetlinewidth{\acadpgflinewidth}
  \pgftransformshift{\pgfpointxy{75.0}{0.0}}

  \begin{pgfscope}
    \pgfpathmoveto{\pgfpointxy{200.0}{250.0}}
    \pgfpathlineto{\pgfpointxy{275.0}{50.0}}
    \pgfusepath{stroke}
  \end{pgfscope}
  \begin{pgfscope}
    \pgfpathmoveto{\pgfpointxy{275.0}{50.0}}
    \pgfpathlineto{\pgfpointxy{125.0}{50.0}}
    \pgfusepath{stroke}
  \end{pgfscope}
  \begin{pgfscope}
    \pgfpathmoveto{\pgfpointxy{125.0}{50.0}}
    \pgfpathlineto{\pgfpointxy{200.0}{250.0}}
    \pgfusepath{stroke}
  \end{pgfscope}
  \begin{pgfscope}
    \pgfpathmoveto{\pgfpointxy{400.0}{250.0}}
    \pgfpathlineto{\pgfpointxy{325.0}{50.0}}
    \pgfusepath{stroke}
  \end{pgfscope}
  \begin{pgfscope}
    \pgfpathmoveto{\pgfpointxy{325.0}{50.0}}
    \pgfpathlineto{\pgfpointxy{475.0}{50.0}}
    \pgfusepath{stroke}
  \end{pgfscope}
  \begin{pgfscope}
    \pgfpathmoveto{\pgfpointxy{475.0}{50.0}}
    \pgfpathlineto{\pgfpointxy{400.0}{250.0}}
    \pgfusepath{stroke}
  \end{pgfscope}
  \begin{pgfscope}
    \pgfpathmoveto{\pgfpointxy{300.0}{350.0}}
    \pgfpathlineto{\pgfpointxy{200.0}{250.0}}
    \pgfusepath{stroke}
  \end{pgfscope}
  \begin{pgfscope}
    \pgfpathmoveto{\pgfpointxy{300.0}{350.0}}
    \pgfpathlineto{\pgfpointxy{400.0}{250.0}}
    \pgfusepath{stroke}
  \end{pgfscope}
  \begin{pgfscope}
    \pgfsetfillcolor{black}
    \pgfpathellipse{\pgfpointxy{300.0}{350.0}}{\pgfpointxy{8.0}{0.0}}{\pgfpointxy{0.0}{8.0}}
    \pgfusepath{fill,stroke}
  \end{pgfscope}
  \begin{pgfscope}
    \pgfsetfillcolor{black}
    \pgfpathellipse{\pgfpointxy{200.0}{250.0}}{\pgfpointxy{8.0}{0.0}}{\pgfpointxy{0.0}{8.0}}
    \pgfusepath{fill,stroke}
  \end{pgfscope}
  \begin{pgfscope}
    \pgfsetfillcolor{black}
    \pgfpathellipse{\pgfpointxy{400.0}{250.0}}{\pgfpointxy{8.0}{0.0}}{\pgfpointxy{0.0}{8.0}}
    \pgfusepath{fill,stroke}
  \end{pgfscope}
  \pgftext[bottom,at={\pgfpointxy{300.0}{370.0}}]{$+$}
  \pgftext[bottom,right,at={\pgfpointxy{242.0}{308.0}}]{$\iota_0$}
  \pgftext[bottom,left,at={\pgfpointxy{358.0}{308.0}}]{$\iota_{n}$}
  \pgftext[bottom,at={\pgfpointxy{200.0}{62.0}}]{$\tau_0$}
  \pgftext[bottom,at={\pgfpointxy{400.0}{62.0}}]{$\tau_{n}$}
  \pgftext[right,at={\pgfpointxy{138.0}{250.0}}]{$\sigma =$}
  \pgftext[at={\pgfpointxy{300.0}{250.0}}]{$\cdots$}
\end{pgfpicture}
  \end{center}
  let us denote this tree as $\sigma = \tau_0 + \ldots + \tau_n$;
  the chain it encodes is $\|\sigma\| = \|\tau_0\| + \ldots + \|\tau_{n}\|$;
\item
  for $\tau_k \in \bigcup_{i < m} \calA_i$, $k \in \omega$,
  let $\tau$, resp.\ $\tau^*$, be a tree whose root is labelled by $\omega$, resp.\ $\omega^*$,
  edges going out of the root are labelled by and ordered as $\omega$, resp.\ $\omega^*$,
  and each edge labelled by $k \in \omega$ leads to a subtree isomorphic to $\tau_k$, $k \in \omega$:
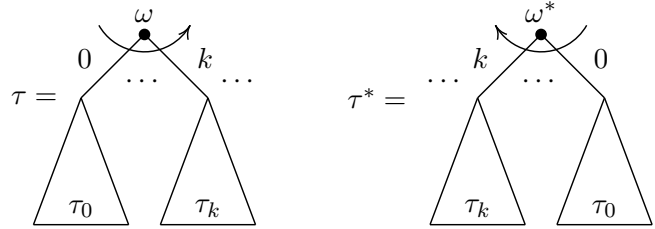
  \begin{center}
\begin{pgfpicture}
  \pgfsetxvec{\pgfpoint{\acadpgfunit}{0pt}}
  \pgfsetyvec{\pgfpoint{0pt}{\acadpgfunit}}
  \pgfsetlinewidth{\acadpgflinewidth}
  \pgftransformshift{\pgfpointxy{50.0}{25.0}}

  \begin{pgfscope}
    \pgfpathmoveto{\pgfpointxy{125.0}{250.0}}
    \pgfpathlineto{\pgfpointxy{200.0}{50.0}}
    \pgfusepath{stroke}
  \end{pgfscope}
  \begin{pgfscope}
    \pgfpathmoveto{\pgfpointxy{200.0}{50.0}}
    \pgfpathlineto{\pgfpointxy{50.0}{50.0}}
    \pgfusepath{stroke}
  \end{pgfscope}
  \begin{pgfscope}
    \pgfpathmoveto{\pgfpointxy{50.0}{50.0}}
    \pgfpathlineto{\pgfpointxy{125.0}{250.0}}
    \pgfusepath{stroke}
  \end{pgfscope}
  \begin{pgfscope}
    \pgfpathmoveto{\pgfpointxy{325.0}{250.0}}
    \pgfpathlineto{\pgfpointxy{250.0}{50.0}}
    \pgfusepath{stroke}
  \end{pgfscope}
  \begin{pgfscope}
    \pgfpathmoveto{\pgfpointxy{250.0}{50.0}}
    \pgfpathlineto{\pgfpointxy{400.0}{50.0}}
    \pgfusepath{stroke}
  \end{pgfscope}
  \begin{pgfscope}
    \pgfpathmoveto{\pgfpointxy{400.0}{50.0}}
    \pgfpathlineto{\pgfpointxy{325.0}{250.0}}
    \pgfusepath{stroke}
  \end{pgfscope}
  \begin{pgfscope}
    \pgfpathmoveto{\pgfpointxy{225.0}{350.0}}
    \pgfpathlineto{\pgfpointxy{125.0}{250.0}}
    \pgfusepath{stroke}
  \end{pgfscope}
  \begin{pgfscope}
    \pgfpathmoveto{\pgfpointxy{225.0}{350.0}}
    \pgfpathlineto{\pgfpointxy{325.0}{250.0}}
    \pgfusepath{stroke}
  \end{pgfscope}
  \begin{pgfscope}
    \pgfpathmoveto{\pgfpointxy{153.106}{364.461}}
    \pgfpatharcaxes{209.62}{330.38}{\pgfpointxy{82.7012}{0.0}}{\pgfpointxy{0.0}{82.7012}}
    \pgfusepath{stroke}
  \end{pgfscope}
  \begin{pgfscope}
    \pgfpathmoveto{\pgfpointxy{278.88}{350.178}}
    \pgfpatharcaxes{286.44}{330.38}{\pgfpointxy{30.7241}{0.0}}{\pgfpointxy{0.0}{30.7241}}
    \pgfusepath{stroke}
  \end{pgfscope}
  \begin{pgfscope}
    \pgfpathmoveto{\pgfpointxy{296.894}{364.461}}
    \pgfpatharcaxes{150.38}{194.319}{\pgfpointxy{30.7241}{0.0}}{\pgfpointxy{0.0}{30.7241}}
    \pgfusepath{stroke}
  \end{pgfscope}
  \begin{pgfscope}
    \pgfpathmoveto{\pgfpointxy{750.0}{250.0}}
    \pgfpathlineto{\pgfpointxy{825.0}{50.0}}
    \pgfusepath{stroke}
  \end{pgfscope}
  \begin{pgfscope}
    \pgfpathmoveto{\pgfpointxy{825.0}{50.0}}
    \pgfpathlineto{\pgfpointxy{675.0}{50.0}}
    \pgfusepath{stroke}
  \end{pgfscope}
  \begin{pgfscope}
    \pgfpathmoveto{\pgfpointxy{675.0}{50.0}}
    \pgfpathlineto{\pgfpointxy{750.0}{250.0}}
    \pgfusepath{stroke}
  \end{pgfscope}
  \begin{pgfscope}
    \pgfpathmoveto{\pgfpointxy{950.0}{250.0}}
    \pgfpathlineto{\pgfpointxy{875.0}{50.0}}
    \pgfusepath{stroke}
  \end{pgfscope}
  \begin{pgfscope}
    \pgfpathmoveto{\pgfpointxy{875.0}{50.0}}
    \pgfpathlineto{\pgfpointxy{1025.0}{50.0}}
    \pgfusepath{stroke}
  \end{pgfscope}
  \begin{pgfscope}
    \pgfpathmoveto{\pgfpointxy{1025.0}{50.0}}
    \pgfpathlineto{\pgfpointxy{950.0}{250.0}}
    \pgfusepath{stroke}
  \end{pgfscope}
  \begin{pgfscope}
    \pgfpathmoveto{\pgfpointxy{850.0}{350.0}}
    \pgfpathlineto{\pgfpointxy{750.0}{250.0}}
    \pgfusepath{stroke}
  \end{pgfscope}
  \begin{pgfscope}
    \pgfpathmoveto{\pgfpointxy{850.0}{350.0}}
    \pgfpathlineto{\pgfpointxy{950.0}{250.0}}
    \pgfusepath{stroke}
  \end{pgfscope}
  \begin{pgfscope}
    \pgfpathmoveto{\pgfpointxy{778.106}{364.461}}
    \pgfpatharcaxes{209.62}{330.38}{\pgfpointxy{82.7012}{0.0}}{\pgfpointxy{0.0}{82.7012}}
    \pgfusepath{stroke}
  \end{pgfscope}
  \begin{pgfscope}
    \pgfpathmoveto{\pgfpointxy{778.106}{364.461}}
    \pgfpatharcaxes{209.62}{253.56}{\pgfpointxy{30.7241}{0.0}}{\pgfpointxy{0.0}{30.7241}}
    \pgfusepath{stroke}
  \end{pgfscope}
  \begin{pgfscope}
    \pgfpathmoveto{\pgfpointxy{781.166}{341.677}}
    \pgfpatharcaxes{-14.3193}{29.6201}{\pgfpointxy{30.7241}{0.0}}{\pgfpointxy{0.0}{30.7241}}
    \pgfusepath{stroke}
  \end{pgfscope}
  \begin{pgfscope}
    \pgfsetfillcolor{black}
    \pgfpathellipse{\pgfpointxy{225.0}{350.0}}{\pgfpointxy{8.0}{0.0}}{\pgfpointxy{0.0}{8.0}}
    \pgfusepath{fill,stroke}
  \end{pgfscope}
  \begin{pgfscope}
    \pgfsetfillcolor{black}
    \pgfpathellipse{\pgfpointxy{850.0}{350.0}}{\pgfpointxy{8.0}{0.0}}{\pgfpointxy{0.0}{8.0}}
    \pgfusepath{fill,stroke}
  \end{pgfscope}
  \pgftext[bottom,at={\pgfpointxy{225.0}{370.0}}]{$\omega$}
  \pgftext[bottom,at={\pgfpointxy{125.0}{62.0}}]{$\tau_0$}
  \pgftext[bottom,at={\pgfpointxy{325.0}{62.0}}]{$\tau_k$}
  \pgftext[right,at={\pgfpointxy{88.0}{250.0}}]{$\tau =$}
  \pgftext[at={\pgfpointxy{225.0}{275.0}}]{$\cdots$}
  \pgftext[at={\pgfpointxy{375.0}{275.0}}]{$\cdots$}
  \pgftext[bottom,right,at={\pgfpointxy{142.0}{297.461}}]{$0$}
  \pgftext[bottom,left,at={\pgfpointxy{308.0}{297.461}}]{$k$}
  \pgftext[bottom,at={\pgfpointxy{850.0}{370.0}}]{$\omega^*$}
  \pgftext[bottom,at={\pgfpointxy{750.0}{62.0}}]{$\tau_k$}
  \pgftext[bottom,at={\pgfpointxy{950.0}{62.0}}]{$\tau_0$}
  \pgftext[right,at={\pgfpointxy{638.0}{250.0}}]{$\tau^* =$}
  \pgftext[at={\pgfpointxy{850.0}{275.0}}]{$\cdots$}
  \pgftext[at={\pgfpointxy{700.0}{275.0}}]{$\cdots$}
  \pgftext[bottom,right,at={\pgfpointxy{767.0}{297.461}}]{$k$}
  \pgftext[bottom,left,at={\pgfpointxy{933.0}{297.461}}]{$0$}
\end{pgfpicture}
  \end{center}
  let us denote the tree as $\tau = \sum_{k \in \omega} \tau_k$, resp.\ $\tau^* = \sum_{k \in \omega^*} \tau_k$;
  the chain it encodes is $\|\tau\| = \sum_{k \in \omega} \|\tau_k\|$, resp.\ $\|\tau^*\| = \sum_{k \in \omega^*} \|\tau_k\|$.
\end{itemize}
Then put
\begin{align*}
  \calA_m =
  &\textstyle \Big\{\sum_{k \in \omega} \tau_k : \tau_k \in \bigcup_{i < m} \calA_i \text{ and } \|\tau_0\| \hookrightarrow \|\tau_1\| \hookrightarrow \ldots\Big\}\\
  &\textstyle \cup \Big\{\sum_{k \in \omega^*} \tau_k : \tau_k \in \bigcup_{i < m} \calA_i \text{ and } \|\tau_0\| \hookrightarrow \|\tau_1\| \hookrightarrow \ldots\Big\}\\
  &\textstyle \cup
    \begin{array}[t]{l@{\,}l}
      \Big\{\sum_{k \in \omega} (\tau_{k0} + \ldots +  \tau_{kn}) : & n \in \NN, \tau_{kj} \in \bigcup_{i < m} \calA_i \text{ and } \\
                                                                    & \|\tau_{0j}\| \hookrightarrow \|\tau_{1j}\| \hookrightarrow \ldots \text{ for all } j\Big\}
    \end{array}\\
  &\textstyle \cup
    \begin{array}[t]{l@{\,}l}
      \Big\{\sum_{k \in \omega^*} (\tau_{k0} + \ldots +  \tau_{kn}) : & n \in \NN, \tau_{kj} \in \bigcup_{i < m} \calA_i \text{ and } \\
                                                                      & \|\tau_{0j}\| \hookrightarrow \|\tau_{1j}\| \hookrightarrow \ldots \text{ for all } j\Big\}.
    \end{array}
\end{align*}
and let 
$$
  \calA = \bigcup_{m \in \omega} \calA_m.
$$
Furthermore, let $\calS$ be the set of trees defined as ``finite sums of trees from~$\calA$'':
$$
  \calS = \calA \union \{\tau_0 + \ldots + \tau_n : n \in \NN \text{ and } \tau_0, \ldots, \tau_n \in \calA\}.
$$

\begin{LEM}\label{fbrd-scat-spec.lem.AI}
  A chain $S$ is an additively indecomposable countable scattered chain of finite Hausdorff rank if and only if there is
  a tree $\tau \in \calA$ such that $S \equiv \|\tau\|$.
\end{LEM}
\begin{proof}
  $(\Rightarrow)$
  This direction follows straightforwardly from Theorem~\ref{fbrd-scat-spec.thm.laver-2}
  by induction on the Hausdorff rank of $S$.
  
  $(\Leftarrow)$ Take any $\tau \in \calA$.
  By construction $\|\tau\|$ is a countable scattered chain of finite Hausdorff rank,
  and $\equiv$ clearly preserves these properties. Let us show by induction on $m$ that
  $\|\tau\|$ is additively indecomposable for every $\tau \in \calA_m$, $m \in \omega$.
  The case $m = 0$ is trivial. Assume that the claim is true for all $j < m$ and take any $\tau \in \calA_m$.
  Let us only consider the possibility where
  $\tau = \sum_{k \in \omega} (\tau_{k0} + \ldots +  \tau_{kn})$ for some $n \in \NN$ and trees
  $\tau_{kj} \in \bigcup_{i < m} \calA_i$, $\0 \le j \le n$, $k \in \omega$, such that
  $$
    \|\tau_{0j}\| \hookrightarrow \|\tau_{1j}\| \hookrightarrow \|\tau_{2j}\| \hookrightarrow \ldots \text{\quad for all\quad} 0 \le j \le n.
  $$
  (The remaining cases follow by similar arguments.)
  For notational convenience let $\sigma_k = \tau_{k0} + \ldots +  \tau_{kn}$, $k \in \omega$.
  It is then obvious that
  \begin{equation}\label{fbrd-scat-spec.eq.1}
    \|\sigma_{0}\| \hookrightarrow \|\sigma_{1}\| \hookrightarrow \|\sigma_{2}\| \hookrightarrow \ldots.
  \end{equation}
  To show that $\|\tau\|$ is additively indecomposable, let $f : \|\tau\| \hookrightarrow A + B$ be an embedding and assume
  that $\|\tau\| \not\hookrightarrow A$. Then there is an $x \in \|\tau\|$ such that $f(x) \in B$.
  Take $i \in \NN$ so that $x \in \|\sigma_i\|$. Then
  \begin{equation}\label{fbrd-scat-spec.eq.2}
    \|\sigma_{i+1}\| \hookrightarrow B.
  \end{equation}
  Therefore,
  $$
    \|\tau\| = \sum_{k \in \omega} \|\sigma_k\| \hookrightarrow \sum_{k \in \omega} \|\sigma_{i + 1 + k}\| \hookrightarrow B
  $$
  where the existence of the first embedding follows from~\eqref{fbrd-scat-spec.eq.1}, while
  the existence of the second embedding follows from~\eqref{fbrd-scat-spec.eq.2} and the fact that $\|\tau\| \hookrightarrow A + B$.
\end{proof}

\begin{LEM}
  A chain $S$ is a countable scattered chain of finite Hausdorff rank if and only if there is a tree $\sigma \in \calS$ such that
  $S \equiv \|\sigma\|$.
\end{LEM}
\begin{proof}
  Directly from Lemma~\ref{fbrd-scat-spec.lem.AI} and Theorem~\ref{fbrd-scat-spec.thm.laver-1}.
\end{proof}

Let $V \subseteq \omega$ be an infinite subset of $\omega$, let $U = V \union \{\iota_n : n \in \omega\}$ and let
$\tau \in \calS$ be arbitrary. Recall that $\restr \tau U$ denotes the subtree of $\tau$ induced by all the branches whose
edge labels belong to $U$. The particular structure of $U$ ensures that the infinite sums in $\tau$ are restricted
so that $\sum_{k \in \omega}$ becomes $\sum_{k \in V}$, and similarly for $\sum_{k \in \omega^*}$.
Thus, we define $\|\restr \tau U\|$ as follows:
\begin{itemize}
\item
  if $\tau = \tau_0 + \ldots + \tau_n$ then $\|\restr\tau U\| = \|\restr{\tau_0}{U}\| + \ldots + \|\restr{\tau_n}{U}\|$;
\item
  if $\tau = \sum_{k \in \omega} \tau_k$ then $\|\restr\tau U\| = \sum_{k \in V} \|\restr{\tau_k}{U}\|$, and
  analogously in case $\tau = \sum_{k \in \omega^*} \tau_k$.
\end{itemize}
Consequently,
\begin{itemize}
\item
  if $\tau = \sum_{k \in \omega} (\tau_{k0} + \ldots +  \tau_{kn})$ then
  $\|\restr\tau U\| = \sum_{k \in V} (\|\restr{\tau_{k0}}{U}\| + \ldots +  \|\restr{\tau_{kn}}{U}\|)$,
  and analogously in case $\tau = \sum_{k \in \omega^*} (\tau_{k0} + \ldots +  \tau_{kn})$.
\end{itemize}
It is obvious that $\|\restr\tau U\| \hookrightarrow \|\tau\|$.

\begin{LEM}\label{fbrd-scat.lem.EQ}
  Let $V \subseteq \omega$ be an infinite subset of $\omega$, let $U = V \union \{\iota_n : n \in \omega\}$ and let
  $\tau \in \calS$ be arbitrary. Then $\|\tau\| \equiv \|\restr\tau U\|$.
\end{LEM}
\begin{proof}
  In view of the last remark, it suffices to show that $\|\tau\| \hookrightarrow \|\restr\tau U\|$ for all $\tau \in \calS$.

  Let us, first, prove the statement in case $\tau \in \calA = \bigcup_{m \in \omega} \calA_m$. The proof is by induction on~$m$.
  Assume that the statement is true for all $\tau \in \bigcup_{i < m} \calA_i$ and take any $\tau \in \calA_m$.
  Let us only consider the possibility where
  $\tau = \sum_{k \in \omega} (\tau_{k0} + \ldots +  \tau_{kn})$ for some $n \in \NN$ and trees
  $\tau_{kj} \in \bigcup_{i < m} \calA_i$, $\0 \le j \le n$, $k \in \omega$, such that
  $$
    \|\tau_{0j}\| \hookrightarrow \|\tau_{1j}\| \hookrightarrow \|\tau_{2j}\| \hookrightarrow \ldots \text{\quad for all\quad} 0 \le j \le n.
  $$
  (The remaining cases follow by similar arguments.)
  For notational convenience let $\sigma_k = \tau_{k0} + \ldots +  \tau_{kn}$, $k \in \omega$.
  It is then obvious that
  \begin{equation}\label{fbrd-scat-spec.eq.3}
    \|\sigma_{0}\| \hookrightarrow \|\sigma_{1}\| \hookrightarrow \|\sigma_{2}\| \hookrightarrow \ldots.
  \end{equation}
  Moreover, by the induction hypothesis,
  \begin{equation}\label{fbrd-scat-spec.eq.4}
    \|\sigma_{k}\| \hookrightarrow \|\restr{\sigma_{k}}{U}\|, \text{\quad for all } k \in \omega.
  \end{equation}
  Let $V = \{v_0 < v_1 < v_2 < \ldots \}$. Then $v_i \ge i$ for all $i \in \omega$, so \eqref{fbrd-scat-spec.eq.3}
  and \eqref{fbrd-scat-spec.eq.4} yield
  $$
    \|\sigma_k\| \hookrightarrow \|\sigma_{v_k}\| \hookrightarrow \|\restr{\sigma_{v_k}}{U}\|, \text{\quad for all } k \in \omega.
  $$
  Therefore,
  $$
    \|\tau\| = \sum_{k \in \omega} \|\sigma_k\| \hookrightarrow \sum_{k \in \omega} \|\restr{\sigma_{v_k}}{U}\|
    = \sum_{\ell \in V} \|\restr{\sigma_{\ell}}{U}\| = \|\restr\tau U\|.
  $$
  Finally, if $\tau \in \calS \setminus \calA$ we have that $\tau = \tau_0 + \ldots + \tau_n$
  for some $n \in \NN$ and $\tau_0, \ldots, \tau_n \in \calA$, and the claim follows immediately from the
  above discussion.
\end{proof}

A tree $\sigma \in \calS$ has \emph{bounded finite sums} if there is an
integer $w \in \NN$ such that $L_e(\sigma) \subseteq \omega \union \{\iota_0, \ldots, \iota_w\}$.
In other words, $\sigma$ is a tree whose finite sums have at most $w + 1$ summands.
A countable scattered chain $S$ of finite Hausdorff rank has \emph{bounded finite sums} if there is a tree $\sigma \in \calS$
with bounded finite sums such that $S \equiv \|\sigma\|$.

We are now going to prove that countable scattered chains of finite Hausdorff rank with bounded finite sums
have finite big Ramsey spectra.
The tool we rely on is the following straightforward consequence of Ramsey's theorem
(for a proof see e.g.~\cite{masul-sobot}):

\begin{THM}\label{fbrd-scat-all.thm.inf-prod-ramsey}
  For every choice of integers $s \ge 1$ and $m_0$, $m_1$, \ldots, $m_{s-1} \ge 1$ there is an integer
  $D = D(s; m_0, m_1, \ldots, m_{s-1})$ such that for every $k \ge 2$ and every coloring
  $
    \chi : \binom\omega{m_0} \times \ldots \times \binom\omega{m_{s-1}} \to k
  $
  there is an infinite $U \subseteq \omega$ satisfying
  $
    \left|\chi\left(\binom U{m_0} \times \ldots \times \binom U{m_{s-1}}\right)\right| \le D.
  $
\end{THM}

\begin{figure}
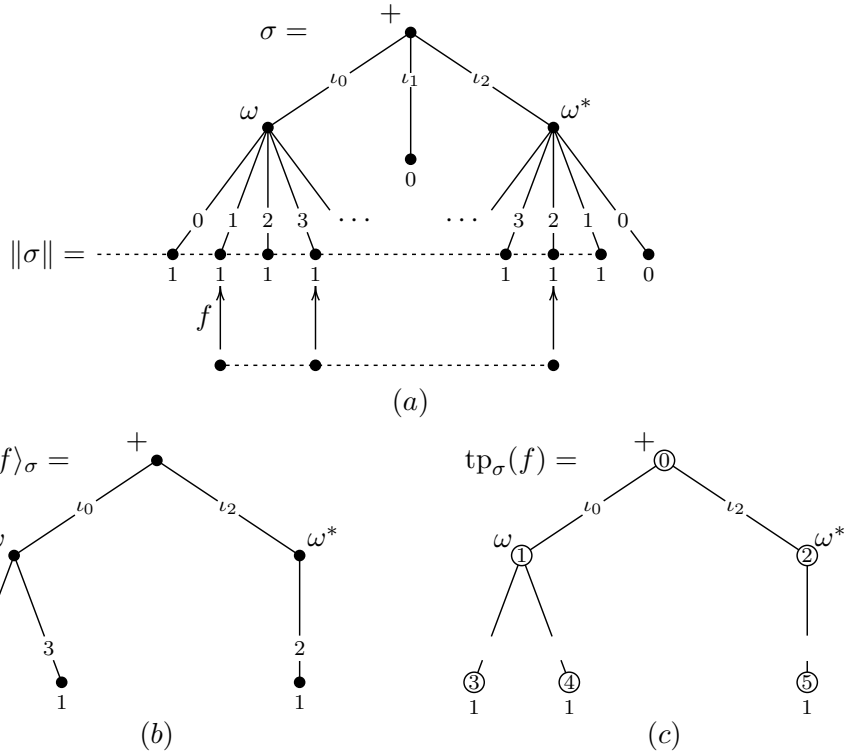

  \centering
  \def\.{\scriptsize}
\begin{pgfpicture}
  \pgfsetxvec{\pgfpoint{\acadpgfunit}{0pt}}
  \pgfsetyvec{\pgfpoint{0pt}{\acadpgfunit}}
  \pgfsetlinewidth{\acadpgflinewidth}
  \pgftransformshift{\pgfpointxy{350.0}{2425.0}}

  \begin{pgfscope}
    \pgfpathmoveto{\pgfpointxy{500.0}{-1150.0}}
    \pgfpathlineto{\pgfpointxy{404.141}{-1213.91}}
    \pgfusepath{stroke}
  \end{pgfscope}
  \begin{pgfscope}
    \pgfpathmoveto{\pgfpointxy{370.859}{-1236.09}}
    \pgfpathlineto{\pgfpointxy{275.0}{-1300.0}}
    \pgfusepath{stroke}
  \end{pgfscope}
  \begin{pgfscope}
    \pgfpathmoveto{\pgfpointxy{725.0}{-1300.0}}
    \pgfpathlineto{\pgfpointxy{627.234}{-1430.36}}
    \pgfusepath{stroke}
  \end{pgfscope}
  \begin{pgfscope}
    \pgfpathmoveto{\pgfpointxy{650.0}{-1500.0}}
    \pgfpathlineto{\pgfpointxy{663.094}{-1465.08}}
    \pgfusepath{stroke}
  \end{pgfscope}
  \begin{pgfscope}
    \pgfpathmoveto{\pgfpointxy{677.139}{-1427.63}}
    \pgfpathlineto{\pgfpointxy{725.0}{-1300.0}}
    \pgfusepath{stroke}
  \end{pgfscope}
  \begin{pgfscope}
    \pgfpathmoveto{\pgfpointxy{725.0}{-1300.0}}
    \pgfpathlineto{\pgfpointxy{725.0}{-1426.36}}
    \pgfusepath{stroke}
  \end{pgfscope}
  \begin{pgfscope}
    \pgfpathmoveto{\pgfpointxy{725.0}{-1466.36}}
    \pgfpathlineto{\pgfpointxy{725.0}{-1500.0}}
    \pgfusepath{stroke}
  \end{pgfscope}
  \begin{pgfscope}
    \pgfpathmoveto{\pgfpointxy{800.0}{-1500.0}}
    \pgfpathlineto{\pgfpointxy{786.906}{-1465.08}}
    \pgfusepath{stroke}
  \end{pgfscope}
  \begin{pgfscope}
    \pgfpathmoveto{\pgfpointxy{772.861}{-1427.63}}
    \pgfpathlineto{\pgfpointxy{725.0}{-1300.0}}
    \pgfusepath{stroke}
  \end{pgfscope}
  \begin{pgfscope}
    \pgfpathmoveto{\pgfpointxy{725.0}{-1300.0}}
    \pgfpathlineto{\pgfpointxy{822.766}{-1430.36}}
    \pgfusepath{stroke}
  \end{pgfscope}
  \begin{pgfscope}
    \pgfpathmoveto{\pgfpointxy{846.766}{-1462.36}}
    \pgfpathlineto{\pgfpointxy{875.0}{-1500.0}}
    \pgfusepath{stroke}
  \end{pgfscope}
  \begin{pgfscope}
    \pgfpathmoveto{\pgfpointxy{500.0}{-1150.0}}
    \pgfpathlineto{\pgfpointxy{595.859}{-1213.91}}
    \pgfusepath{stroke}
  \end{pgfscope}
  \begin{pgfscope}
    \pgfpathmoveto{\pgfpointxy{629.141}{-1236.09}}
    \pgfpathlineto{\pgfpointxy{725.0}{-1300.0}}
    \pgfusepath{stroke}
  \end{pgfscope}
  \begin{pgfscope}
    \pgfpathmoveto{\pgfpointxy{275.0}{-1300.0}}
    \pgfpathlineto{\pgfpointxy{372.766}{-1430.36}}
    \pgfusepath{stroke}
  \end{pgfscope}
  \begin{pgfscope}
    \pgfpathmoveto{\pgfpointxy{125.0}{-1500.0}}
    \pgfpathlineto{\pgfpointxy{153.234}{-1462.36}}
    \pgfusepath{stroke}
  \end{pgfscope}
  \begin{pgfscope}
    \pgfpathmoveto{\pgfpointxy{177.234}{-1430.36}}
    \pgfpathlineto{\pgfpointxy{275.0}{-1300.0}}
    \pgfusepath{stroke}
  \end{pgfscope}
  \begin{pgfscope}
    \pgfpathmoveto{\pgfpointxy{275.0}{-1300.0}}
    \pgfpathlineto{\pgfpointxy{227.139}{-1427.63}}
    \pgfusepath{stroke}
  \end{pgfscope}
  \begin{pgfscope}
    \pgfpathmoveto{\pgfpointxy{213.094}{-1465.08}}
    \pgfpathlineto{\pgfpointxy{200.0}{-1500.0}}
    \pgfusepath{stroke}
  \end{pgfscope}
  \begin{pgfscope}
    \pgfpathmoveto{\pgfpointxy{275.0}{-1500.0}}
    \pgfpathlineto{\pgfpointxy{275.0}{-1466.36}}
    \pgfusepath{stroke}
  \end{pgfscope}
  \begin{pgfscope}
    \pgfpathmoveto{\pgfpointxy{275.0}{-1426.36}}
    \pgfpathlineto{\pgfpointxy{275.0}{-1300.0}}
    \pgfusepath{stroke}
  \end{pgfscope}
  \begin{pgfscope}
    \pgfpathmoveto{\pgfpointxy{275.0}{-1300.0}}
    \pgfpathlineto{\pgfpointxy{322.861}{-1427.63}}
    \pgfusepath{stroke}
  \end{pgfscope}
  \begin{pgfscope}
    \pgfpathmoveto{\pgfpointxy{336.906}{-1465.08}}
    \pgfpathlineto{\pgfpointxy{350.0}{-1500.0}}
    \pgfusepath{stroke}
  \end{pgfscope}
  \begin{pgfscope}
    \pgfsetdash{{1.5pt}{2pt}}{0pt}
    \pgfpathmoveto{\pgfpointxy{800.0}{-1500.0}}
    \pgfpathlineto{\pgfpointxy{0.0}{-1500.0}}
    \pgfusepath{stroke}
  \end{pgfscope}
  \begin{pgfscope}
    \pgfpathmoveto{\pgfpointxy{725.0}{-1650.0}}
    \pgfpathlineto{\pgfpointxy{725.0}{-1550.0}}
    \pgfusepath{stroke}
  \end{pgfscope}
  \begin{pgfscope}
    \pgfpathmoveto{\pgfpointxy{718.971}{-1572.5}}
    \pgfpatharcaxes{-30.0}{0.0}{\pgfpointxy{45.0}{0.0}}{\pgfpointxy{0.0}{45.0}}
    \pgfusepath{stroke}
  \end{pgfscope}
  \begin{pgfscope}
    \pgfpathmoveto{\pgfpointxy{725.0}{-1550.0}}
    \pgfpatharcaxes{180.0}{210.0}{\pgfpointxy{45.0}{0.0}}{\pgfpointxy{0.0}{45.0}}
    \pgfusepath{stroke}
  \end{pgfscope}
  \begin{pgfscope}
    \pgfpathmoveto{\pgfpointxy{350.0}{-1650.0}}
    \pgfpathlineto{\pgfpointxy{350.0}{-1550.0}}
    \pgfusepath{stroke}
  \end{pgfscope}
  \begin{pgfscope}
    \pgfpathmoveto{\pgfpointxy{343.971}{-1572.5}}
    \pgfpatharcaxes{-30.0}{0.0}{\pgfpointxy{45.0}{0.0}}{\pgfpointxy{0.0}{45.0}}
    \pgfusepath{stroke}
  \end{pgfscope}
  \begin{pgfscope}
    \pgfpathmoveto{\pgfpointxy{350.0}{-1550.0}}
    \pgfpatharcaxes{180.0}{210.0}{\pgfpointxy{45.0}{0.0}}{\pgfpointxy{0.0}{45.0}}
    \pgfusepath{stroke}
  \end{pgfscope}
  \begin{pgfscope}
    \pgfpathmoveto{\pgfpointxy{200.0}{-1650.0}}
    \pgfpathlineto{\pgfpointxy{200.0}{-1550.0}}
    \pgfusepath{stroke}
  \end{pgfscope}
  \begin{pgfscope}
    \pgfpathmoveto{\pgfpointxy{193.971}{-1572.5}}
    \pgfpatharcaxes{-30.0}{0.0}{\pgfpointxy{45.0}{0.0}}{\pgfpointxy{0.0}{45.0}}
    \pgfusepath{stroke}
  \end{pgfscope}
  \begin{pgfscope}
    \pgfpathmoveto{\pgfpointxy{200.0}{-1550.0}}
    \pgfpatharcaxes{180.0}{210.0}{\pgfpointxy{45.0}{0.0}}{\pgfpointxy{0.0}{45.0}}
    \pgfusepath{stroke}
  \end{pgfscope}
  \begin{pgfscope}
    \pgfsetdash{{1.5pt}{2pt}}{0pt}
    \pgfpathmoveto{\pgfpointxy{725.0}{-1675.0}}
    \pgfpathlineto{\pgfpointxy{200.0}{-1675.0}}
    \pgfusepath{stroke}
  \end{pgfscope}
  \begin{pgfscope}
    \pgfpathmoveto{\pgfpointxy{500.0}{-1150.0}}
    \pgfpathlineto{\pgfpointxy{500.0}{-1213.91}}
    \pgfusepath{stroke}
  \end{pgfscope}
  \begin{pgfscope}
    \pgfpathmoveto{\pgfpointxy{500.0}{-1236.09}}
    \pgfpathlineto{\pgfpointxy{500.0}{-1350.0}}
    \pgfusepath{stroke}
  \end{pgfscope}
  \begin{pgfscope}
    \pgfpathmoveto{\pgfpointxy{100.0}{-1825.0}}
    \pgfpathlineto{\pgfpointxy{4.14101}{-1888.91}}
    \pgfusepath{stroke}
  \end{pgfscope}
  \begin{pgfscope}
    \pgfpathmoveto{\pgfpointxy{-29.141}{-1911.09}}
    \pgfpathlineto{\pgfpointxy{-125.0}{-1975.0}}
    \pgfusepath{stroke}
  \end{pgfscope}
  \begin{pgfscope}
    \pgfpathmoveto{\pgfpointxy{325.0}{-1975.0}}
    \pgfpathlineto{\pgfpointxy{325.0}{-2101.36}}
    \pgfusepath{stroke}
  \end{pgfscope}
  \begin{pgfscope}
    \pgfpathmoveto{\pgfpointxy{325.0}{-2141.36}}
    \pgfpathlineto{\pgfpointxy{325.0}{-2175.0}}
    \pgfusepath{stroke}
  \end{pgfscope}
  \begin{pgfscope}
    \pgfpathmoveto{\pgfpointxy{100.0}{-1825.0}}
    \pgfpathlineto{\pgfpointxy{195.859}{-1888.91}}
    \pgfusepath{stroke}
  \end{pgfscope}
  \begin{pgfscope}
    \pgfpathmoveto{\pgfpointxy{229.141}{-1911.09}}
    \pgfpathlineto{\pgfpointxy{325.0}{-1975.0}}
    \pgfusepath{stroke}
  \end{pgfscope}
  \begin{pgfscope}
    \pgfpathmoveto{\pgfpointxy{-125.0}{-1975.0}}
    \pgfpathlineto{\pgfpointxy{-172.861}{-2102.63}}
    \pgfusepath{stroke}
  \end{pgfscope}
  \begin{pgfscope}
    \pgfpathmoveto{\pgfpointxy{-186.906}{-2140.08}}
    \pgfpathlineto{\pgfpointxy{-200.0}{-2175.0}}
    \pgfusepath{stroke}
  \end{pgfscope}
  \begin{pgfscope}
    \pgfpathmoveto{\pgfpointxy{-125.0}{-1975.0}}
    \pgfpathlineto{\pgfpointxy{-77.1393}{-2102.63}}
    \pgfusepath{stroke}
  \end{pgfscope}
  \begin{pgfscope}
    \pgfpathmoveto{\pgfpointxy{-63.0944}{-2140.08}}
    \pgfpathlineto{\pgfpointxy{-50.0}{-2175.0}}
    \pgfusepath{stroke}
  \end{pgfscope}
  \begin{pgfscope}
    \pgfpathmoveto{\pgfpointxy{886.687}{-1833.88}}
    \pgfpathlineto{\pgfpointxy{804.141}{-1888.91}}
    \pgfusepath{stroke}
  \end{pgfscope}
  \begin{pgfscope}
    \pgfpathmoveto{\pgfpointxy{770.859}{-1911.09}}
    \pgfpathlineto{\pgfpointxy{688.313}{-1966.12}}
    \pgfusepath{stroke}
  \end{pgfscope}
  \begin{pgfscope}
    \pgfpathmoveto{\pgfpointxy{1125.0}{-1991.0}}
    \pgfpathlineto{\pgfpointxy{1125.0}{-2101.36}}
    \pgfusepath{stroke}
  \end{pgfscope}
  \begin{pgfscope}
    \pgfpathmoveto{\pgfpointxy{1125.0}{-2141.36}}
    \pgfpathlineto{\pgfpointxy{1125.0}{-2159.0}}
    \pgfusepath{stroke}
  \end{pgfscope}
  \begin{pgfscope}
    \pgfpathmoveto{\pgfpointxy{913.313}{-1833.88}}
    \pgfpathlineto{\pgfpointxy{995.859}{-1888.91}}
    \pgfusepath{stroke}
  \end{pgfscope}
  \begin{pgfscope}
    \pgfpathmoveto{\pgfpointxy{1029.14}{-1911.09}}
    \pgfpathlineto{\pgfpointxy{1111.69}{-1966.12}}
    \pgfusepath{stroke}
  \end{pgfscope}
  \begin{pgfscope}
    \pgfpathmoveto{\pgfpointxy{669.382}{-1989.98}}
    \pgfpathlineto{\pgfpointxy{627.139}{-2102.63}}
    \pgfusepath{stroke}
  \end{pgfscope}
  \begin{pgfscope}
    \pgfpathmoveto{\pgfpointxy{613.094}{-2140.08}}
    \pgfpathlineto{\pgfpointxy{605.618}{-2160.02}}
    \pgfusepath{stroke}
  \end{pgfscope}
  \begin{pgfscope}
    \pgfpathmoveto{\pgfpointxy{680.618}{-1989.98}}
    \pgfpathlineto{\pgfpointxy{722.861}{-2102.63}}
    \pgfusepath{stroke}
  \end{pgfscope}
  \begin{pgfscope}
    \pgfpathmoveto{\pgfpointxy{736.906}{-2140.08}}
    \pgfpathlineto{\pgfpointxy{744.382}{-2160.02}}
    \pgfusepath{stroke}
  \end{pgfscope}
  \begin{pgfscope}
    \pgfpathellipse{\pgfpointxy{675.0}{-1975.0}}{\pgfpointxy{16.0}{0.0}}{\pgfpointxy{0.0}{16.0}}
    \pgfusepath{stroke}
  \end{pgfscope}
  \begin{pgfscope}
    \pgfpathellipse{\pgfpointxy{900.0}{-1825.0}}{\pgfpointxy{16.0}{0.0}}{\pgfpointxy{0.0}{16.0}}
    \pgfusepath{stroke}
  \end{pgfscope}
  \begin{pgfscope}
    \pgfpathellipse{\pgfpointxy{600.0}{-2175.0}}{\pgfpointxy{16.0}{0.0}}{\pgfpointxy{0.0}{16.0}}
    \pgfusepath{stroke}
  \end{pgfscope}
  \begin{pgfscope}
    \pgfpathellipse{\pgfpointxy{750.0}{-2175.0}}{\pgfpointxy{16.0}{0.0}}{\pgfpointxy{0.0}{16.0}}
    \pgfusepath{stroke}
  \end{pgfscope}
  \begin{pgfscope}
    \pgfpathellipse{\pgfpointxy{1125.0}{-1975.0}}{\pgfpointxy{16.0}{0.0}}{\pgfpointxy{0.0}{16.0}}
    \pgfusepath{stroke}
  \end{pgfscope}
  \begin{pgfscope}
    \pgfpathellipse{\pgfpointxy{1125.0}{-2175.0}}{\pgfpointxy{16.0}{0.0}}{\pgfpointxy{0.0}{16.0}}
    \pgfusepath{stroke}
  \end{pgfscope}
  \begin{pgfscope}
    \pgfsetfillcolor{black}
    \pgfpathellipse{\pgfpointxy{125.0}{-1500.0}}{\pgfpointxy{8.0}{0.0}}{\pgfpointxy{0.0}{8.0}}
    \pgfusepath{fill,stroke}
  \end{pgfscope}
  \begin{pgfscope}
    \pgfsetfillcolor{black}
    \pgfpathellipse{\pgfpointxy{200.0}{-1500.0}}{\pgfpointxy{8.0}{0.0}}{\pgfpointxy{0.0}{8.0}}
    \pgfusepath{fill,stroke}
  \end{pgfscope}
  \begin{pgfscope}
    \pgfsetfillcolor{black}
    \pgfpathellipse{\pgfpointxy{275.0}{-1500.0}}{\pgfpointxy{8.0}{0.0}}{\pgfpointxy{0.0}{8.0}}
    \pgfusepath{fill,stroke}
  \end{pgfscope}
  \begin{pgfscope}
    \pgfsetfillcolor{black}
    \pgfpathellipse{\pgfpointxy{350.0}{-1500.0}}{\pgfpointxy{8.0}{0.0}}{\pgfpointxy{0.0}{8.0}}
    \pgfusepath{fill,stroke}
  \end{pgfscope}
  \begin{pgfscope}
    \pgfsetfillcolor{black}
    \pgfpathellipse{\pgfpointxy{875.0}{-1500.0}}{\pgfpointxy{8.0}{0.0}}{\pgfpointxy{0.0}{8.0}}
    \pgfusepath{fill,stroke}
  \end{pgfscope}
  \begin{pgfscope}
    \pgfsetfillcolor{black}
    \pgfpathellipse{\pgfpointxy{800.0}{-1500.0}}{\pgfpointxy{8.0}{0.0}}{\pgfpointxy{0.0}{8.0}}
    \pgfusepath{fill,stroke}
  \end{pgfscope}
  \begin{pgfscope}
    \pgfsetfillcolor{black}
    \pgfpathellipse{\pgfpointxy{725.0}{-1500.0}}{\pgfpointxy{8.0}{0.0}}{\pgfpointxy{0.0}{8.0}}
    \pgfusepath{fill,stroke}
  \end{pgfscope}
  \begin{pgfscope}
    \pgfsetfillcolor{black}
    \pgfpathellipse{\pgfpointxy{650.0}{-1500.0}}{\pgfpointxy{8.0}{0.0}}{\pgfpointxy{0.0}{8.0}}
    \pgfusepath{fill,stroke}
  \end{pgfscope}
  \begin{pgfscope}
    \pgfsetfillcolor{black}
    \pgfpathellipse{\pgfpointxy{275.0}{-1300.0}}{\pgfpointxy{8.0}{0.0}}{\pgfpointxy{0.0}{8.0}}
    \pgfusepath{fill,stroke}
  \end{pgfscope}
  \begin{pgfscope}
    \pgfsetfillcolor{black}
    \pgfpathellipse{\pgfpointxy{725.0}{-1300.0}}{\pgfpointxy{8.0}{0.0}}{\pgfpointxy{0.0}{8.0}}
    \pgfusepath{fill,stroke}
  \end{pgfscope}
  \begin{pgfscope}
    \pgfsetfillcolor{black}
    \pgfpathellipse{\pgfpointxy{500.0}{-1150.0}}{\pgfpointxy{8.0}{0.0}}{\pgfpointxy{0.0}{8.0}}
    \pgfusepath{fill,stroke}
  \end{pgfscope}
  \begin{pgfscope}
    \pgfsetfillcolor{black}
    \pgfpathellipse{\pgfpointxy{200.0}{-1675.0}}{\pgfpointxy{8.0}{0.0}}{\pgfpointxy{0.0}{8.0}}
    \pgfusepath{fill,stroke}
  \end{pgfscope}
  \begin{pgfscope}
    \pgfsetfillcolor{black}
    \pgfpathellipse{\pgfpointxy{350.0}{-1675.0}}{\pgfpointxy{8.0}{0.0}}{\pgfpointxy{0.0}{8.0}}
    \pgfusepath{fill,stroke}
  \end{pgfscope}
  \begin{pgfscope}
    \pgfsetfillcolor{black}
    \pgfpathellipse{\pgfpointxy{725.0}{-1675.0}}{\pgfpointxy{8.0}{0.0}}{\pgfpointxy{0.0}{8.0}}
    \pgfusepath{fill,stroke}
  \end{pgfscope}
  \begin{pgfscope}
    \pgfsetfillcolor{black}
    \pgfpathellipse{\pgfpointxy{500.0}{-1350.0}}{\pgfpointxy{8.0}{0.0}}{\pgfpointxy{0.0}{8.0}}
    \pgfusepath{fill,stroke}
  \end{pgfscope}
  \begin{pgfscope}
    \pgfsetfillcolor{black}
    \pgfpathellipse{\pgfpointxy{-200.0}{-2175.0}}{\pgfpointxy{8.0}{0.0}}{\pgfpointxy{0.0}{8.0}}
    \pgfusepath{fill,stroke}
  \end{pgfscope}
  \begin{pgfscope}
    \pgfsetfillcolor{black}
    \pgfpathellipse{\pgfpointxy{-50.0}{-2175.0}}{\pgfpointxy{8.0}{0.0}}{\pgfpointxy{0.0}{8.0}}
    \pgfusepath{fill,stroke}
  \end{pgfscope}
  \begin{pgfscope}
    \pgfsetfillcolor{black}
    \pgfpathellipse{\pgfpointxy{325.0}{-2175.0}}{\pgfpointxy{8.0}{0.0}}{\pgfpointxy{0.0}{8.0}}
    \pgfusepath{fill,stroke}
  \end{pgfscope}
  \begin{pgfscope}
    \pgfsetfillcolor{black}
    \pgfpathellipse{\pgfpointxy{-125.0}{-1975.0}}{\pgfpointxy{8.0}{0.0}}{\pgfpointxy{0.0}{8.0}}
    \pgfusepath{fill,stroke}
  \end{pgfscope}
  \begin{pgfscope}
    \pgfsetfillcolor{black}
    \pgfpathellipse{\pgfpointxy{325.0}{-1975.0}}{\pgfpointxy{8.0}{0.0}}{\pgfpointxy{0.0}{8.0}}
    \pgfusepath{fill,stroke}
  \end{pgfscope}
  \begin{pgfscope}
    \pgfsetfillcolor{black}
    \pgfpathellipse{\pgfpointxy{100.0}{-1825.0}}{\pgfpointxy{8.0}{0.0}}{\pgfpointxy{0.0}{8.0}}
    \pgfusepath{fill,stroke}
  \end{pgfscope}
  \pgftext[at={\pgfpointxy{387.5}{-1225.0}}]{\.$\iota_0$}
  \pgftext[at={\pgfpointxy{612.5}{-1225.0}}]{\.$\iota_2$}
  \pgftext[at={\pgfpointxy{834.766}{-1446.36}}]{\.0}
  \pgftext[at={\pgfpointxy{779.883}{-1446.36}}]{\.1}
  \pgftext[at={\pgfpointxy{725.0}{-1446.36}}]{\.2}
  \pgftext[at={\pgfpointxy{670.117}{-1446.36}}]{\.3}
  \pgftext[right,at={\pgfpointxy{615.234}{-1446.36}}]{$\cdots$}
  \pgftext[at={\pgfpointxy{165.234}{-1446.36}}]{\.0}
  \pgftext[at={\pgfpointxy{220.117}{-1446.36}}]{\.1}
  \pgftext[at={\pgfpointxy{275.0}{-1446.36}}]{\.2}
  \pgftext[at={\pgfpointxy{329.883}{-1446.36}}]{\.3}
  \pgftext[left,at={\pgfpointxy{384.766}{-1446.36}}]{$\cdots$}
  \pgftext[bottom,right,at={\pgfpointxy{486.362}{-1136.32}}]{$+$}
  \pgftext[bottom,right,at={\pgfpointxy{261.298}{-1286.39}}]{$\omega$}
  \pgftext[bottom,left,at={\pgfpointxy{737.21}{-1285.2}}]{$\omega^*$}
  \pgftext[right,at={\pgfpointxy{-12.0}{-1500.0}}]{$\|\sigma\| =$}
  \pgftext[right,at={\pgfpointxy{338.0}{-1150.0}}]{$\sigma = $}
  \pgftext[right,at={\pgfpointxy{188.0}{-1600.0}}]{$f$}
  \pgftext[top,at={\pgfpointxy{500.0}{-1712.0}}]{$(a)$}
  \pgftext[top,at={\pgfpointxy{100.0}{-2237.0}}]{$(b)$}
  \pgftext[top,at={\pgfpointxy{900.0}{-2237.0}}]{$(c)$}
  \pgftext[at={\pgfpointxy{500.0}{-1225.0}}]{\.$\iota_1$}
  \pgftext[at={\pgfpointxy{-12.5}{-1900.0}}]{\.$\iota_0$}
  \pgftext[at={\pgfpointxy{212.5}{-1900.0}}]{\.$\iota_2$}
  \pgftext[at={\pgfpointxy{325.0}{-2121.36}}]{\.2}
  \pgftext[at={\pgfpointxy{-179.883}{-2121.36}}]{\.1}
  \pgftext[at={\pgfpointxy{-70.1169}{-2121.36}}]{\.3}
  \pgftext[bottom,right,at={\pgfpointxy{86.3617}{-1811.32}}]{$+$}
  \pgftext[bottom,right,at={\pgfpointxy{-138.702}{-1961.39}}]{$\omega$}
  \pgftext[bottom,left,at={\pgfpointxy{337.21}{-1960.2}}]{$\omega^*$}
  \pgftext[right,at={\pgfpointxy{-37.0}{-1825.0}}]{$\tree{f}_\sigma = $}
  \pgftext[at={\pgfpointxy{787.5}{-1900.0}}]{\.$\iota_0$}
  \pgftext[at={\pgfpointxy{1012.5}{-1900.0}}]{\.$\iota_2$}
  \pgftext[bottom,right,at={\pgfpointxy{886.362}{-1811.32}}]{$+$}
  \pgftext[bottom,right,at={\pgfpointxy{661.298}{-1961.39}}]{$\omega$}
  \pgftext[bottom,left,at={\pgfpointxy{1137.21}{-1960.2}}]{$\omega^*$}
  \pgftext[right,at={\pgfpointxy{763.0}{-1825.0}}]{$\tp_\sigma(f) = $}
  \pgftext[at={\pgfpointxy{900.0}{-1825.0}}]{\.0}
  \pgftext[at={\pgfpointxy{675.0}{-1975.0}}]{\.1}
  \pgftext[at={\pgfpointxy{1125.0}{-1975.0}}]{\.2}
  \pgftext[at={\pgfpointxy{600.0}{-2175.0}}]{\.3}
  \pgftext[at={\pgfpointxy{750.0}{-2175.0}}]{\.4}
  \pgftext[at={\pgfpointxy{1125.0}{-2175.0}}]{\.5}
  \pgftext[top,at={\pgfpointxy{875.0}{-1520.0}}]{\.0}
  \pgftext[top,at={\pgfpointxy{800.0}{-1520.0}}]{\.1}
  \pgftext[top,at={\pgfpointxy{725.0}{-1520.0}}]{\.1}
  \pgftext[top,at={\pgfpointxy{650.0}{-1520.0}}]{\.1}
  \pgftext[top,at={\pgfpointxy{125.0}{-1520.0}}]{\.1}
  \pgftext[top,at={\pgfpointxy{200.0}{-1520.0}}]{\.1}
  \pgftext[top,at={\pgfpointxy{275.0}{-1520.0}}]{\.1}
  \pgftext[top,at={\pgfpointxy{350.0}{-1520.0}}]{\.1}
  \pgftext[top,at={\pgfpointxy{500.0}{-1370.0}}]{\.0}
  \pgftext[top,at={\pgfpointxy{-200.0}{-2195.0}}]{\.1}
  \pgftext[top,at={\pgfpointxy{-50.0}{-2195.0}}]{\.1}
  \pgftext[top,at={\pgfpointxy{325.0}{-2195.0}}]{\.1}
  \pgftext[top,at={\pgfpointxy{600.0}{-2203.0}}]{\.1}
  \pgftext[top,at={\pgfpointxy{750.0}{-2203.0}}]{\.1}
  \pgftext[top,at={\pgfpointxy{1125.0}{-2203.0}}]{\.1}
\end{pgfpicture}
  \caption{$(a)$ A tree $\sigma$, the chain $\|\sigma\|$ it encodes and an embedding $f : 3 \hookrightarrow \|\sigma\|$;
           $(b)$ The tree $\tree{f}_\sigma$ that corresponds to $f$;
           $(c)$ The type $\tp_\sigma(f)$ of the embedding $f$.}
  \label{fbrd-scat-all.fig.tree-type}
\end{figure}

Take any tree $\sigma \in \calS$.
Every embedding $f : n \hookrightarrow \|\sigma\|$, $n \in \NN$, corresponds to a subtree of $\sigma$
induced by branches $[v_0, \ell_i]_{\|\sigma\|}$, $i < n$, where $v_0$ is the root of $\sigma$
and $\ell_i$ is the leaf of $\sigma$ that corresponds to $f(i)$, $i < n$.
Let us denote this subtree of $\sigma$ by $\tree{f}_\sigma$. Clearly, $\tree{f}_\sigma$ has $n$ leaves
and its height is the same as the height of $\sigma$ (see Fig.~\ref{fbrd-scat-all.fig.tree-type}~$(a)$ and $(b)$).

Assume, now, that $\tree{f}_\sigma$ has $p$ vertices. If we replace the vertex set of
$\tree{f}_\sigma$ by $\{0, 1, \ldots, p - 1\}$ so that the usual ordering of the integers agrees with the BFS-ordering of the
new tree, and then erase only those edge labels that come from $\omega$, the resulting labelled ordered rooted tree on the set of vertices $\{0, 1, \ldots, p-1\}$
will be referred to as the \emph{type of $f$} and will be denoted by $\tp_\sigma(f)$.
Fig.~\ref{fbrd-scat-all.fig.tree-type}~$(c)$ depicts the type of the embedding $f$ given in
Fig.~\ref{fbrd-scat-all.fig.tree-type}~$(a)$.

A finite labelled ordered rooted tree $\tau$ is an
\emph{$(n, \sigma)$-type} if $\tau = \tp_\sigma(g)$ for some embedding $g : n \hookrightarrow \|\sigma\|$.
Therefore, for all $n \in \NN$ and all $\sigma \in \calS$ each
$(n, \sigma)$-type is a labelled ordered rooted tree with the following properties:
\begin{itemize}
\item
  its set of vertices is $\{0, 1, \ldots, p-1\}$ for some $p \in \NN$
  and the BFS-order of the tree coincides with the usual ordering of the integers
  (hence 0 is the root of the tree);
\item
  it has $n$ leaves and its height is the height of $\sigma$;
\item
  its leaves are labelled by 1, while other vertices are labelled by~$+$, $\omega$ or~$\omega^*$; and
\item
  its edges going out of vertices labelled by~$+$ are labelled by~$\iota_0$, $\iota_1$, \ldots, while other edges are not labelled.
\end{itemize}
(Note that a labelled ordered rooted tree with the above properties need not be an $(n, \sigma)$-type.)
The following is now obvious:

\begin{LEM}
  Given an $n \in \NN$ and a $\sigma \in \calS$ with bounded finite sums,
  there are only finitely many $(n, \sigma)$-types.
\end{LEM}

For an $(n, \sigma)$-type $\tau$ and a set of edge labels $U \subseteq \omega \union \{\iota_n : n \in \omega\}$ let
$$
  \Emb_\tau(n, \restr \sigma U) = \{ f \in \Emb(n, \|\restr \sigma U\|) : \tp_\sigma(f) = \tau \}.
$$

\begin{LEM}\label{fbrd-scat-all.lem.1}
  Let $V \subseteq \omega$ be an infinite subset of~$\omega$, let $U = V \cup \{\iota_n : n \in \omega\}$,
  let $n \in \NN$ and let $\sigma \in \calS$ have bounded finite sums.
  For every $(n, \sigma)$-type $\tau$ there is a $D_\tau \in \NN$
  such that for every $k \ge 2$ and every coloring
  $
    \chi : \Emb_\tau(n, \restr \sigma U) \to k
  $
  there is an infinite $V' \subseteq V$ such that for $U' = V' \union \{\iota_n : n \in \omega\}$:
  $$
    |\chi(\Emb_\tau(n, \restr \sigma {U'})| \le D_\tau.
  $$
\end{LEM}
\begin{proof}
  Assume, first, that no vertex of $\tau$ is labelled by either $\omega$ or $\omega^*$.
  Since $\sigma$ has bounded finite sums, there is an integer $w \in \NN$ such that
  $L_e(\sigma) \subseteq \omega \union \{\iota_0, \ldots, \iota_w\}$.
  Let $\high(\sigma) = h$. Then there are at most $(w + 1)^h$ branches in $\sigma$ that do not
  pass through a vertex labelled by $\omega$ or $\omega^*$. There are at most $\binom{(w + 1)^h}{n}$ subtrees of
  $\sigma$ induced by choosing some $n$ of those branches. Therefore, for any coloring
  $\chi : \Emb_\tau(n, \restr \sigma U) \to k$ and any $U'$ it must be the case that
  $|\chi(\Emb_\tau(n, \restr \sigma {U'})| \le \binom{(w + 1)^h}{n}$.

  For the rest of the proof assume that at least one vertex of $\tau$ is labelled by $\omega$ or $\omega^*$.
  Let $\ell_0 < \ell_1 < \ldots < \ell_{s-1}$ be all the vertices of~$\tau$ labelled by $\omega$ or $\omega^*$.
  Let $m_i = |\out_\tau(\ell_i)|$, $i < s$, and let
  $D_\tau = D(s; m_0, m_1, \ldots, m_{s-1})$ be the number provided by Theorem~\ref{fbrd-scat-all.thm.inf-prod-ramsey}.

  Take any $f \in \Emb_\tau(n, \restr \sigma U)$ and let $(v_0, v_1, \ldots, v_{p-1})$ be the vertex
  set of $\tree{f}_{\restr \sigma U}$ ordered by the BFS-order of $\tree{f}_{\restr \sigma U}$.
  Since $\tp_\sigma(f) = \tau$, the only vertices
  in $\tree{f}_{\restr \sigma U}$ labelled by $\omega$ or $\omega^*$ are $v_{\ell_0}$, $v_{\ell_1}$, \ldots, $v_{\ell_{s-1}}$.
  Let $L_f(i) \subseteq V$ be the set of all the labels used to label the edges in $\out(v_{\ell_i})$, $i < s$.
  Clearly, $|L_f(i)| = m_i$, $i < s$.

  By construction, each embedding $f \in \Emb_\tau(n, \restr \sigma U)$ is uniquely determined by the
  sequence $(L_f(0), L_f(1), \ldots, L_f(s-1))$ of
  subsets of $V$ of sizes $m_{0}$, $m_{1}$, \ldots, $m_{s-1}$, respectively.
  Therefore,
  $$
    \Phi : \Emb_\tau(n, \restr \sigma U) \to \binom{V}{m_0} \times \binom{V}{m_1} \times \ldots \times \binom{V}{m_{s-1}}
  $$
  given by
  $$
    \Phi(f) = (L_f(0), L_f(1), \ldots, L_f(s-1))
  $$
  is an injective mapping.
  Now, take any $k \ge 2$ and any coloring
  $\chi : \Emb_\tau(n, \restr \sigma U) \to k$, and define
  $$
    \chi' : \binom{V}{m_0} \times \ldots \times \binom{V}{m_{s-1}} \to k
  $$
  by
  $$
    \chi'(A_0, A_1, \ldots, A_{s-1}) = \begin{cases}
      \chi(f), & \Phi(f) = (A_0, A_1, \ldots, A_{s-1}),\\
      0, & \text{otherwise}.
    \end{cases}
  $$
  Then by Theorem~\ref{fbrd-scat-all.thm.inf-prod-ramsey} there exists an infinite $V' \subseteq V$ such that
  $$
    \left|\chi'\left(\binom {V'}{m_0} \times \ldots \times \binom {V'}{m_{s-1}}\right)\right| \le D_\tau.
  $$
  Let $U' = V' \union \{\iota_n : n \in \omega\}$. The construction of $\chi'$ ensures that
  $$
    \chi(\Emb_\tau(n, \restr \sigma {U'})) \subseteq \chi'\left(\binom {V'}{m_0} \times \ldots \times \binom {V'}{m_{s-1}}\right),
  $$
  whence $|\chi(\Emb_\tau(n, \restr \sigma {U'}))| \le D_\tau$.
\end{proof}

\begin{THM}\label{fbrd-scat-all.thm.finite-rank}
  Let $S$ be a countable scattered chain such that $r_H(S) < \omega$. Assume additionally that
  $S$ has bounded finite sums. Then $\spec(S)$ is finite.
\end{THM}
\begin{proof}
  Since $S$ countable scattered chain $S$ of finite Hausdorff rank and with bounded finite sums,
  there is a tree $\sigma \in \calS$ with bounded finite sums such that $S \equiv \|\sigma\|$.
  Without loss of generality we may assume that $S = \|\sigma\|$.
  Take any $n \in \NN$ and let us show that $T(n, S)$ is finite.
  Let $\tau_0$, $\tau_1$, \dots, $\tau_{s-1}$ be all the $(n, \sigma)$-types
  and let $D_{\tau_0}$, $D_{\tau_1}$, \dots, $D_{\tau_{s-1}}$ be the integers provided by Lemma~\ref{fbrd-scat-all.lem.1}.
  We are going to show that
  $
    T(n, S) \le \sum_{j < s} D_{\tau_j} < \infty
  $.
  
  Take any $k \ge 2$ and any coloring $\chi : \Emb(n, S) \to k$. Since
  $$
    \Emb(n, S) = \Emb(n, \|\sigma\|) = \bigcup_{j < s} \Emb_{\tau_j}(n, \sigma)
  $$
  we have that $\chi$ is also a coloring of $\Emb_{\tau_j}(n, \sigma)$ for
  all $j < s$. By Lemma~\ref{fbrd-scat-all.lem.1} there is an infinite $V_0 \subseteq \omega$ such that
  for $U_0 = V_0 \union \{\iota_q : q \in \omega\}$,
  $$
    |\chi(\Emb_{\tau_0}(n, \restr{\sigma}{U_0}))| \le D_{\tau_0}.
  $$
  By the same lemma for each $j \in \{1, \ldots, s-1\}$
  we can inductively construct an infinite $V_j \subseteq V_{j-1}$ such that
  for $U_j = V_j \union \{\iota_q : q \in \omega\}$,
  $$
    |\chi(\Emb_{\tau_j}(n, \restr{\sigma}{U_j}))| \le D_{\tau_j}.
  $$
  Then, having in mind that $U_{s-1} \subseteq U_j$ for all $j < s$,
  \begin{align*}
    |\chi(\Emb(n, \|\restr{\sigma}{U_{s-1}}\|))|
    &= \sum_{j < t} |\chi(\Emb_{\tau_j}(n, \restr{\sigma}{U_{s-1}}))|\\
    &\le \sum_{j < t} |\chi(\Emb_{\tau_j}(n, \restr{\sigma}{U_{j}}))| \le \sum_{j < t} D_{\tau_j}.
  \end{align*}
  The theorem now follows from the fact that $\|\sigma\| \equiv \|\restr{\sigma}{U_{s-1}}\|$
  (Lemma~\ref{fbrd-scat.lem.EQ}).
\end{proof}

\begin{COR}
  Let $S$ be a countable scattered chain of finite Hausdorff rank.
  If $T(1, S) = 1$ then $\spec(S)$ is finite.
\end{COR}
\begin{proof}
  Let $S$ be a countable scattered chain of finite Hausdorff rank such that $T(1, S) = 1$.
  According to \cite[Theorem 2.13]{laver-decomposition}, there is a tree $\sigma \in \calA$
  with no finite sums such that $S \equiv \|\sigma\|$.
  (In \cite{laver-decomposition} the order type of such an $S$ is referred to as \emph{hereditarily increasing}).
  But then finite sums in $\sigma$ are bounded, whence
  $\spec(S)$ is finite by Theorem~\ref{fbrd-scat-all.thm.finite-rank}.
\end{proof}

\paragraph{Open problem.}
We have seen in Section~\ref{fbrd-scat-all.sec.inf-rank} that countable scattered chains of infinite
Hausdorff rank do not have finite big Ramsey spectra. On the other hand, we have just proved in
Theorem~\ref{fbrd-scat-all.thm.finite-rank} that countable scattered chains of finite
Hausdorff rank and \emph{with bounded finite sums} have finite big Ramsey spectra.
At the moment, we are unable to resolve the remaining case of countable scattered chains of finite
Hausdorff rank whose finite sums are not bounded.

\section*{Data availability statement}

Data sharing not applicable to this article as no datasets were generated or analysed during the current study.

\section*{Acknowledgements}

The author would like to thank F.~Galvin and S.~Todor\v cevi\'c for their many insightful observations
during the preparation of \cite{masul-sobot} which paved the road for this paper.

The author would also like to thank the two anonymous referees for the careful reading of the paper and many
constructive comments that brought much needed clarity to the last two sections of the paper.

The author gratefully acknowledges the financial support of the Ministry of Education, Science and Technological Development
of the Republic of Serbia (Grant No.\ 451-03-9/2021-14/200125).

\end{document}